\documentclass[12pt]{amsart}
\usepackage{txfonts}      %{article} was 12pt latex e
\usepackage{amssymb}
\usepackage{eucal}
\usepackage{amsmath}
\usepackage{amscd}
\usepackage{xcolor}
\usepackage{multicol}
\usepackage[all]{xy}           %xypic macro for latex2.09
\usepackage{graphicx}
\usepackage{color}
\usepackage{colordvi}
\usepackage{xspace}
\usepackage{tikz}
\usepackage{makecell}
\usepackage{appendix}
\usepackage{amsthm}
\usepackage[compress]{cite}

\usepackage{ifpdf}
\ifpdf
\usepackage[colorlinks,final,backref=page,hyperindex]{hyperref}
\else
\usepackage[colorlinks,final,backref=page,hyperindex,hypertex]{hyperref}
\fi

\usepackage[active]{srcltx} %SRC Specials for DVI Searching

\begin{document}

%%%%%%%%%%%%%%%%%%%%%%%% Statements

\newtheorem{thm}{Theorem}[section]
\newtheorem{lem}[thm]{Lemma}
\newtheorem{cor}[thm]{Corollary}
\newtheorem{pro}[thm]{Proposition}
\theoremstyle{definition}
\newtheorem{defi}[thm]{Definition}
\newtheorem{ex}[thm]{Example}
\newtheorem{rmk}[thm]{Remark}
\newtheorem{pdef}[thm]{Proposition-Definition}
\newtheorem{condition}[thm]{Condition}

\renewcommand{\labelenumi}{{\rm(\alph{enumi})}}
\renewcommand{\theenumi}{\alph{enumi}}

\newcommand {\emptycomment}[1]{} %to remove paragraphs

\newcommand{\nc}{\newcommand}
\newcommand{\delete}[1]{}

\nc{\tred}[1]{\textcolor{red}{#1}}
\nc{\tblue}[1]{\textcolor{blue}{#1}}
\nc{\tgreen}[1]{\textcolor{green}{#1}}
\nc{\tpurple}[1]{\textcolor{purple}{#1}}
\nc{\tgray}[1]{\textcolor{gray}{#1}}
\nc{\torg}[1]{\textcolor{orange}{#1}}
\nc{\tmag}[1]{\textcolor{magenta}}
\nc{\btred}[1]{\textcolor{red}{\bf #1}}
\nc{\btblue}[1]{\textcolor{blue}{\bf #1}}
\nc{\btgreen}[1]{\textcolor{green}{\bf #1}}
\nc{\btpurple}[1]{\textcolor{purple}{\bf #1}}

%\nc{\revise}[1]{\textcolor{blue}{#1}}

%%%%%%%% new symbols

\nc{\tforall}{\ \ \text{for all }} \nc{\hatot}{\,\widehat{\otimes}
\,} \nc{\complete}{completed\xspace} \nc{\wdhat}[1]{\widehat{#1}}

\nc{\ts}{\mathfrak{p}} \nc{\mts}{c_{(i)}\ot d_{(j)}}

\nc{\NA}{{\bf NA}} \nc{\LA}{{\bf Lie}} \nc{\CLA}{{\bf CLA}}

\nc{\cybe}{CYBE\xspace} \nc{\nybe}{NYBE\xspace}
\nc{\ccybe}{CCYBE\xspace}

\nc{\ndend}{pre-Novikov\xspace} \nc{\calb}{\mathcal{B}}
\nc{\rk}{\mathrm{r}}
\newcommand{\g}{\mathfrak g}
\newcommand{\h}{\mathfrak h}
\newcommand{\pf}{\noindent{$Proof$.}\ }
\newcommand{\frkg}{\mathfrak g}
\newcommand{\frkh}{\mathfrak h}
\newcommand{\Id}{\rm{Id}}
\newcommand{\gl}{\mathfrak {gl}}
\newcommand{\ad}{\mathrm{ad}}
\newcommand{\add}{\frka\frkd}
\newcommand{\frka}{\mathfrak a}
\newcommand{\frkb}{\mathfrak b}
\newcommand{\frkc}{\mathfrak c}
\newcommand{\frkd}{\mathfrak d}
\newcommand {\comment}[1]{{\marginpar{*}\scriptsize\textbf{Comments:} #1}}

\nc{\vspa}{\vspace{-.1cm}} \nc{\vspb}{\vspace{-.2cm}}
\nc{\vspc}{\vspace{-.3cm}} \nc{\vspd}{\vspace{-.4cm}}
\nc{\vspe}{\vspace{-.5cm}}

%%%%%%%%%%%%%%%%%%%%%%% old symbols

\nc{\disp}[1]{\displaystyle{#1}}
\nc{\bin}[2]{ (_{\stackrel{\scs{#1}}{\scs{#2}}})}  %binomial coeff
\nc{\binc}[2]{ \left (\!\! \begin{array}{c} \scs{#1}\\
    \scs{#2} \end{array}\!\! \right )}  %binomial coeff
\nc{\bincc}[2]{  \left ( {\scs{#1} \atop
    \vspace{-.5cm}\scs{#2}} \right )}  %binomial coeff
\nc{\ot}{\otimes} \nc{\sot}{{\scriptstyle{\ot}}}
\nc{\otm}{\overline{\ot}}
\nc{\ola}[1]{\stackrel{#1}{\la}}%${\Bbb Z}$

\nc{\scs}[1]{\scriptstyle{#1}} \nc{\mrm}[1]{{\rm #1}}

\nc{\dirlim}{\displaystyle{\lim_{\longrightarrow}}\,}
\nc{\invlim}{\displaystyle{\lim_{\longleftarrow}}\,}

\nc{\bfk}{{\bf k}} \nc{\bfone}{{\bf 1}} \nc{\rpr}{\circ}
%\nc{\apr}{\cdot}
\nc{\dpr}{{\tiny\diamond}} \nc{\rprpm}{{\rpr}}

\nc{\calao}{{\mathcal A}} \nc{\cala}{{\mathcal A}}
\nc{\calc}{{\mathcal C}} \nc{\cald}{{\mathcal D}}
\nc{\cale}{{\mathcal E}} \nc{\calf}{{\mathcal F}}
\nc{\calfr}{{{\mathcal F}^{\,r}}} \nc{\calfo}{{\mathcal F}^0}
\nc{\calfro}{{\mathcal F}^{\,r,0}} \nc{\oF}{\overline{F}}
\nc{\calg}{{\mathcal G}} \nc{\calh}{{\mathcal H}}
\nc{\cali}{{\mathcal I}} \nc{\calj}{{\mathcal J}}
\nc{\call}{{\mathcal L}} \nc{\calm}{{\mathcal M}}
\nc{\caln}{{\mathcal N}} \nc{\calo}{{\mathcal O}}
\nc{\calp}{{\mathcal P}} \nc{\calq}{{\mathcal Q}}
\nc{\calr}{{\mathcal R}} \nc{\calt}{{\mathcal T}}
\nc{\caltr}{{\mathcal T}^{\,r}} \nc{\calu}{{\mathcal U}}
\nc{\calv}{{\mathcal V}} \nc{\calw}{{\mathcal W}}
\nc{\calx}{{\mathcal X}} \nc{\CA}{\mathcal{A}}

%%%%%%%%%%%%%%%%%%  frak fonts
\nc{\fraka}{{\mathfrak a}} \nc{\frakB}{{\mathfrak B}}
\nc{\frakb}{{\mathfrak b}} \nc{\frakd}{{\mathfrak d}}
\nc{\oD}{\overline{D}} \nc{\frakF}{{\mathfrak F}}
\nc{\frakg}{{\mathfrak g}} \nc{\frakm}{{\mathfrak m}}
\nc{\frakM}{{\mathfrak M}} \nc{\frakMo}{{\mathfrak M}^0}
\nc{\frakp}{{\mathfrak p}} \nc{\frakS}{{\mathfrak S}}
\nc{\frakSo}{{\mathfrak S}^0} \nc{\fraks}{{\mathfrak s}}
\nc{\os}{\overline{\fraks}} \nc{\frakT}{{\mathfrak T}}
\nc{\oT}{\overline{T}}
%\nc{\frakx}{{\mathfrak x}}
\nc{\frakX}{{\mathfrak X}} \nc{\frakXo}{{\mathfrak X}^0}
\nc{\frakx}{{\mathbf x}}
%\nc{\frakTxo}{{\frakTx}^0}
\nc{\frakTx}{\frakT}      %All rooted trees, correspond to \ncsha(X)
\nc{\frakTa}{\frakT^a}        % rooted trees for \ncsha(A)
\nc{\frakTxo}{\frakTx^0}   % rooted trees for \ncshao(X)
\nc{\caltao}{\calt^{a,0}}   % rooted trees for \ncshao(A)
\nc{\ox}{\overline{\frakx}} \nc{\fraky}{{\mathfrak y}}
\nc{\frakz}{{\mathfrak z}} \nc{\oX}{\overline{X}}

%%%%%%%%%%%%%%%%%%%%%%%%%%%%%%%%%%%%%%%%%%%%%%%%%%%%%%%%%%%%%%%%%%

%%%%%%%%%%%%%%%%%%%%%%%%%%%%%%%%%%%%%%%%%%%%%%%%%%%%%%%%%%%%%%%%%%

\title[Non-abelian extensions of Lie triple systems and Wells exact sequences]{Non-abelian extensions of Lie triple systems and Wells exact sequences}

\author{Qinxiu Sun}
\address{Qinxiu Sun, Department of Mathematics, Zhejiang University of Science and Technology, Hangzhou, 310023} \email{qxsun@126.com}

\author{Shuangjian Guo*}
\address{Shuangjian Guo, Corresponding author, School of Mathematics and Statistics, Guizhou University of Finance and Economics, Guiyang, 550025}
         \email{shuangjianguo@126.com}



\begin{abstract}

In this paper, we investigate non-abelian
extensions and inducibility of pairs of automorphisms of Lie triple systems.
 First, we introduce non-abelian cohomology groups and classify the non-abelian extensions in terms of non-abelian cohomology groups.
 Next, we characterize the non-abelian extensions using Maurer-Cartan elements. Furthermore, we explore the inducibility of pairs
 of automorphisms and derive the analog Wells exact sequences under the circumstance of Lie triple systems. Finally,
we state the previous results under the context of abelian extensions of Lie triple systems.

\end{abstract}

\maketitle

\vspace{-1.2cm}

\tableofcontents

\vspace{-1.2cm}

\allowdisplaybreaks

\section{Introduction}
Lie triple systems originated from Cartan's studies of Riemannian
geometry, in which he employed his classification of the real simple
Lie algebras to classify symmetric spaces. The role played by Lie
triple systems in the theory of symmetric spaces is like that of Lie
algebras in the theory of Lie groups. Lie triple systems were
investigated from the algebraic point of view by  Jacobson in
\cite{51} and studied later by Lister in \cite{27}. Subsequently,
Lie triple systems have gained broad attention and a number of works
have been gained, see \cite{00,54,53} and their references.

Extensions of some mathematical object are useful to understand the
underlying structure. The non-abelian extension is a relatively
general one among various extensions (e.g. central extensions,
abelian extensions, non-abelian extensions etc.). Non-abelian
extensions were first developed by Eilenberg and Maclane \cite{010},
which leaded to the low dimensional non-abelian group cohomology.
Then numerous works have been devoted to non-abelian extensions of
various kinds of algebras, such as Lie algebras, associative
algebras, Lie groups, pre-Lie algebras, Leibniz algebras, 3-Lie algebras, Rota-Baxter
Lie algebras, conformal algebras and Rota-Baxter Leibniz algebras,
see \cite{013,08,09,201,014,015,016,048,027,047,029} and their references.
But little is known about the non-abelian extension of Lie triple systems.
This was our first motivation for writing the present paper.

Our second motivation is to explore the inducibility of pairs
of Lie triple system automorphisms. The inducibility of a pair of automorphisms has close relationship with extensions of algebras.
 Such study originated from extensions of abstract groups \cite{67}. This theme was subsequently
extended to various algebraic structures, such as Lie algebras, Rota-Baxter
Lie algebras, Rota-Baxter groups, associative conformal algebras and Lie-Yamaguti algebras, see \cite{67,045,032,047,021,026,09,60,004}
and references therein. As byproducts, the Wells short exact sequences were derived for various kinds of algebras \cite{021, 08,026,045,09,047,004},
 which connected the relative automorphism groups and
the non-abelian second cohomology groups. Motived by these results and combined with our study of
non-abelian extension of Lie triple systems, naturally we study inducibility of pairs of Lie triple system
automorphisms and derive the Wells short exact sequences in the context of non-abelian extensions of Lie triple systems.

The paper is organized as follows. In Section 2, we recall basic
knowledge of Lie triple systems. In Section 3, we
investigate non-abelian extensions and classified the non-abelian
extensions using non-abelian 3-cocycles. In Section 4, we characterize
non-abelian extensions in terms of Maurer-Cartan elements. In Section 5, we address the inducibility of
pairs of Lie triple system automorphisms and give some equivalent conditions
In Section 6, we consider the analogue of
the Wells short exact sequences under the case of Lie triple systems. In Section 7,
we explain the previous results in the context of abelian extensions of Lie triple systems.

Throughout the paper, let $k$ be a field. Unless otherwise
specified, all vector spaces and algebras are over $k$.

\setlength{\baselineskip}{1.25\baselineskip}

%%%%%%%%%%%%%%%%%%%%%%%%%%%%%%%%%%%%%%%%%%%%%%%%%%%%%%%%%%%%%%%%%%%%%%%%%%%%%%%%%%%%%%%%%%%%%%%%%%%%%%%

\section{Preliminary on Lie triple systems}

A Lie triple system is a vector space $\mathfrak g$ together with a
trilinear map $[ \ , \ , \ ]: \mathfrak g\otimes \mathfrak g\otimes
\mathfrak g\longrightarrow \mathfrak g$ satisfying
\begin{equation}\label{DLts1}[x_1,x_1,x_2]=0,\end{equation}
\begin{equation}\label{DLts2}[x_1,x_2,x_3]+[x_2,x_3,x_1]+[x_3,x_1,x_2]=0,\end{equation}
\begin{equation}\label{DLts3}[x_1,x_2,[y_1,y_2,y_3]]=[[x_1,x_2,y_1],y_2,y_3]+[y_1,[x_1,x_2,y_2],y_3]+[y_1,y_2,[x_1,x_2,y_3]],\end{equation}
for all $x_i,y_i\in \mathfrak g ~(i=1,2,3)$.

 Eqs.~~(\ref{DLts1}) and (\ref{DLts2}) yield that
\begin{equation}\label{DLts4}[x_1,x_2,x_3]+[x_2,x_1,x_3]=0.\end{equation}

A subspace $M$ of $\mathfrak g$ is an ideal of $\mathfrak g$ if
$ [M, \mathfrak g,\mathfrak g ] \subseteq M$ and $[\mathfrak g,\mathfrak g, M] \subseteq M$.
 An ideal $M$ of $\mathfrak g$ is said to be an abelian ideal of $\mathfrak g$ if
 $[M, M,\mathfrak g ] = [\mathfrak g,M, M ] =0$.

A representation of a Lie triple system $\mathfrak g$ consists of a
vector space $V$ together with a bilinear map $\theta:\mathfrak
g\wedge \mathfrak g\longrightarrow \mathfrak{gl}(V)$ satisfying
\begin{equation}\label{RLts1}\theta(x_3, x_4)\theta(x_1, x_2)-\theta(x_2, x_4)\theta(x_1,
x_3)-\theta(x_1,[x_2,x_3,x_4])+D_{\theta}(x_2, x_3)\theta(x_1,
x_4)=0,\end{equation}
\begin{equation}\label{RLts2}
\theta(x_3, x_4)D_{\theta}(x_1, x_2)-D_{\theta}(x_1, x_2)\theta(x_3,
x_4)+\theta([x_1,x_2,x_3],x_4)+\theta(x_3,[x_1,x_2,x_4])=0,\end{equation}
for any $x_i\in \mathfrak g~ (i=1,2,3,4)$, where
\begin{equation}\label{RLts3}
D_{\theta}(x_1,
x_2)=\theta(x_2, x_1)-\theta(x_1, x_2).\end{equation}

One can refer to \cite{21,22, 27, 34,35,135} for more information on Lie
triple systems and representations.

Let $\mathfrak g$ be a Lie triple system and $(V,\theta)$ be a
representation of it. Denote the $(2n+1)$-cochains group of
$\mathfrak g$ with coefficients in $V$ by $ C^{2n+1}(\mathfrak
g,V)$, where $f\in C^{2n+1}(\mathfrak g,V)$ is a multilinear map
$f:\mathfrak g\times \cdot\cdot\cdot \times \mathfrak
g\longrightarrow V$ satisfying
\begin{equation}\label{CO1}f(x_1,\cdot\cdot\cdot,x_{2n-2},x,x,y)=0\end{equation} and
\begin{equation}\label{CO2}f(x_1,\cdot\cdot\cdot,x_{2n-2},x,y,z)+f(x_1,\cdot\cdot\cdot,x_{2n-2},y,z,x)+f(x_1,\cdot\cdot\cdot,x_{2n-2},z,x,y)=0\end{equation}
for all $x_i,x,y,z\in \mathfrak g ~(i=1,\cdots, 2n-2)$ and $ C^{1}(\mathfrak
g,V)=\hbox{Hom}(\mathfrak g,V)$.

 The Yamaguti coboundary operator $\delta:C^{2n-1}(\mathfrak g,V)\longrightarrow C^{2n+1}(\mathfrak g,V)$ is given by
\begin{eqnarray*}&&(\delta f)(x_1,\cdot\cdot\cdot,x_{2n+1})
\\&=&\theta(x_{2n},x_{2n+1})f(x_1,\cdot\cdot\cdot,x_{2n-1})-\theta(x_{2n-1},x_{2n+1})f(x_1,\cdot\cdot\cdot,x_{2n-2},x_{2n})
\\&&+
\sum_{i=1}^{n}(-1)^{i+n}D_{\theta}(x_{2i-1},x_{2i})f(x_1,\cdot\cdot\cdot,x_{2i-2},x_{2i+1},\cdot\cdot\cdot,x_{2n+1})
\\&&+
\sum_{i=1}^{n}\sum_{j=2i+1}^{2n+1}(-1)^{i+n+1}f(x_1,\cdot\cdot\cdot,x_{2i-2},x_{2i+1},\cdot\cdot\cdot,x_{j-1},[x_{2i-1},x_{2i},x_j]
,\cdot\cdot\cdot,x_{2n+1})
\end{eqnarray*}
for any $x_1,x_2,\cdot\cdot\cdot,x_{2n+1}\in \mathfrak g$. Denote the set of all $(2n-1)$-cocycles and $(2n-1)$-coboundaries
respectively by $Z^{2n-1}(\mathfrak g,V)$ and $ B^{2n-1}(\mathfrak g,V)$.
Thus, we can define the $(2n-1)$-cohomology group by
\begin{equation*}
	H^{2n-1}(\mathfrak g,V)=\left\{
	\begin{aligned}
		&Z^{1}(A, V ),&n=1,\\
		&Z^{2n-1}(A,V)/B^{2n-1}(A,V),&n>1.
	\end{aligned}
	\right.
\end{equation*}
See \cite{30, 35} for more details on Yamaguti cohomology.

%%%%%%%%%%%%%%%%%%%%%%%%%%%%%%%%%%%%%%%%%%%%%%%%%%%%%%%%%%%%%%%%%%%%%%%%%%%%%%%%%%%%%%%%%%%%%%%%%%%%%%%

\section{Non-abelian
extensions of Lie triple systems}

In this section, we are devoted to considering non-abelian
extensions and non-abelian 3-cocycles of Lie triple systems.

\begin{defi} Let $\mathfrak g$ and $\mathfrak h$ be two Lie triple systems. A non-abelian
extension of $\mathfrak g$ by $\mathfrak h$ is a Lie triple system $\hat{\mathfrak g}$,
which fits into a short exact sequence of Lie triple systems
$$\mathcal{E}:0\longrightarrow\mathfrak h\stackrel{i}{\longrightarrow} \hat{\mathfrak g}\stackrel{p}{\longrightarrow}\mathfrak g\longrightarrow0.$$
When $\mathfrak h$ is an abelian ideal of $\hat{\mathfrak g}$, the extension $\mathcal{E}$ is called an
abelian extension of $\mathfrak g$ by $\mathfrak h$.
Denote an extension as above simply by $\hat{\mathfrak g}$ or $\mathcal{E}$. A section of $p$
is a linear map $s: \mathfrak g\longrightarrow  \hat{\mathfrak g} $
such that $ps = I_{\mathfrak g}$.
\end{defi}

\begin{defi}
Let $ \hat{\mathfrak g}_1$
and $ \hat{\mathfrak g}_2$
be two non-abelian extensions of $\mathfrak g$ by $\mathfrak h$.
 They are said to be
equivalent if there is a homomorphism of Lie triple systems
$f:\hat{\mathfrak g}_1\longrightarrow \hat{\mathfrak g}_2$ such that
the following commutative diagram holds:
 \begin{equation}\label{Ene1} \xymatrix{
  0 \ar[r] & \mathfrak h\ar@{=}[d] \ar[r]^{i_1} & \hat{\mathfrak g}_1\ar[d]_{f} \ar[r]^{p_1} & \mathfrak g \ar@{=}[d] \ar[r] & 0\\
 0 \ar[r] & \mathfrak h \ar[r]^{i_2} & \hat{\mathfrak g}_2 \ar[r]^{p_2} & \mathfrak g  \ar[r] & 0
 .}\end{equation}
\end{defi}
Denote by $\mathcal{E}_{nab}(\mathfrak g,\mathfrak h)$ the set of all non-abelian extensions of $\mathfrak g$ by $\mathfrak h$.

Next, we define a non-abelian cohomology group and show that
the non-abelian extensions are classified by the non-abelian cohomology groups.

\begin{defi} Let $\mathfrak g$ and $\mathfrak h$ be two Lie triple systems.
A non-abelian 3-cocycle on $\mathfrak g$ with values in
 $\mathfrak h$ is a triple $(\omega,\theta,\rho)$
of maps such that $\omega:\mathfrak g\otimes\mathfrak g\otimes\mathfrak
g\longrightarrow \mathfrak h$ is trilinear, $\theta:\mathfrak g\wedge \mathfrak
g\longrightarrow \mathfrak{gl}(\mathfrak h)$ is bilinear
and $\rho:\mathfrak g\longrightarrow
\mathrm{Hom} (\mathfrak h\wedge \mathfrak h,\mathfrak h)$ is linear,
and the following identities are satisfied for all
$x_i,y_i, x, y, z,w\in \mathfrak g~ (i=1,2,3), a,b,c\in \mathfrak h$,
\begin{equation}\label{Lts02}\omega(x,y,z)+\omega(y,z,x)+\omega(z,x,y)=0,
 \end{equation}
\begin{align}\label{Lts1}&D_{\theta}(x_1,x_2)\omega(y_1,y_2,y_3)+\omega(x_1,x_2,[y_1,y_2,y_3]_{\mathfrak
g})=\omega([x_1,x_2,y_1]_{\mathfrak
g},y_2,y_3)\nonumber\\&+\theta(y_2,y_3)\omega(x_1,x_2,y_1)+\omega(y_1,[x_1,x_2,y_2]_{\mathfrak
g},y_3)-\theta(y_1,y_3)\omega(x_1,x_2,y_2)\nonumber\\&+\omega(y_1,y_2,[x_1,x_2,y_3]_{\mathfrak
g})+D_{\theta}(y_1,y_2)\omega(x_1,x_2,y_3),
 \end{align}
\begin{align}\label{Lts2}&D_{\theta}(x,y)\theta(z,w)a-\theta(z,w)D_{\theta}(x,y)a\nonumber\\=&\theta([x,y,z]_{\mathfrak
g},w)a+\theta(z,[x,y,w]_{\mathfrak g})a-D_{\rho}(w)(\omega(x,y,z),a)-\rho(z)(a,\omega(x,y,w)),
 \end{align}
\begin{equation}\label{Lts3}\theta(x,[y,z,w]_{\mathfrak
g})a-\rho(x)(a,\omega(y,z,w))=\theta(z,w)\theta(x,y)a-\theta(y,w)\theta(x,z)a+D_{\theta}(y,z)\theta(x,w)a,
 \end{equation}
\begin{align}\label{Lts4}D_{\theta}(x,y)\rho(z)(a,b)=&\rho(z)(D_{\theta}(x,y)a,b)+\rho([x,y,z]_{\mathfrak
g})(a,b)\nonumber\\&+[\omega(x,y,z),a,b]_{\mathfrak
h}+\rho(z)(a,D_{\theta}(x,y)b),
 \end{align}
\begin{align}\label{Lts5}D_{\theta}(y,z)\rho(x)(a,b)=\rho(x)(a,D_{\theta}(y,z)b)-\rho(z
)(\theta(x,y)a,b)+\rho(y)(\theta(x,z)a,b),
 \end{align}
\begin{equation}\label{Lts6}\theta(y,z)(\rho(x)(a,b))=\rho(x)(a,\theta(y,z)b)-D_{\rho}(z
)(\theta(x,y)a,b)-\rho(y)(b,\theta(x,z)a),
 \end{equation}
\begin{align}\label{Lts7}D_{\theta}(x,y)([a,b,c]_{\mathfrak
h})=[D_{\theta}(x,y)a,b,c]_{\mathfrak
h}+[a,D_{\theta}(x,y)b,c]_{\mathfrak
h}+[a,b,D_{\theta}(x,y)c]_{\mathfrak h},
 \end{align}
\begin{equation}\label{Lts8}\rho(x)(a,\rho(y)(b,c))=\rho(y)(\rho(x)(a,b),c)+\rho(y)(b,\rho(x)(a,c))
-[\theta(x,y)a,b,c]_{\mathfrak h},
 \end{equation}
\begin{equation}\label{Lts9}[a,b,\theta(x,y)c]_{\mathfrak h}=\theta(x,y)[a,b,c]_{\mathfrak
h}-D_{\rho}(y)(D_{\rho}(x)(a,b),c)-\rho(x)(c,D_{\rho}(y)(a,b)),
 \end{equation}
\begin{equation}\label{Lts10}[a,b,\rho(x)(c,d)]_{\mathfrak h}=\rho(x)([a,b,c]_{\mathfrak
h},d)-[c,D_{\rho}(x)(a,b),d]_{\mathfrak
h}+\rho(x)(c,[a,b,d]_{\mathfrak h}),
 \end{equation}
\begin{equation}\label{Lts11}\rho(x)(a,[b,c,d]_{\mathfrak
h})=[\rho(x)(a,b),c,d]_{\mathfrak h}+[b,\rho(x)(a,c),d]_{\mathfrak
h}+[b,c,\rho(x)(a,d)]_{\mathfrak h},
 \end{equation}
 where
 \begin{equation}\label{Lts01}D_{\theta}(x,y)a-\theta(y,x)a+\theta(x,y)a=0,
~~\rho(x)(a,b)+D_{\rho}(x)(a,b)-\rho(x)(b,a)=0.
 \end{equation}
\end{defi}

\begin{defi} Let $(\omega_1,\theta_1,\rho_1)$ and
$(\omega_2,\theta_2,\rho_2)$ be two
non-abelian 3-cocycles on $\mathfrak g$ with values in
 $\mathfrak h$. They are said to be equivalent if there
exists a linear map $\varphi:\mathfrak g\longrightarrow\mathfrak h$
such that for all $x, y, z\in \mathfrak g$ and $a,b\in \mathfrak h$,
the following equalities hold:
\begin{align}\label{E2}&\omega_1(x,y,z)-\omega_2(x,y,z)=\theta_{2}(x,z)\varphi(y)-D_{\theta_2}(x,y)\varphi(z)
+\rho_2(x)(\varphi(y),\varphi(z))-\theta_{2}(y,z)\varphi(x)
\nonumber\\&+D_{\rho_2}(z)(\varphi(x),\varphi(y))-\rho_2(y)(\varphi(x),\varphi(z))-[\varphi(x),\varphi(y),\varphi(z)]_{\mathfrak
h}+\varphi([x,y,z]_{\mathfrak
g}),
 \end{align}
   \begin{equation}\label{E4}\theta_1(x,y)a-\theta_2(x,y)a=\rho_{2}(x)(a,\varphi(y))-D_{\rho_2}(y)(a,\varphi(x))+[a,\varphi(x),\varphi(y)]_{\mathfrak
h},
 \end{equation}
\begin{equation}\label{E6}\rho_1(x)(a,b)-\rho_2(x)(a,b)=[a,\varphi(x),b]_{\mathfrak
h}.\end{equation}
 \end{defi}
Using Eqs.~~(\ref{Lts01}), (\ref{E4}) and (\ref{E6}), we get
 \begin{equation}\label{E3}D_{\theta_1}(x,y)a-D_{\theta_2}(x,y)a=\rho_2(y)(\varphi(x),a)-\rho_{2}(x)(\varphi(y),a)+[\varphi(x),\varphi(y),a]_{\mathfrak
h},
 \end{equation}
 and
 \begin{equation}\label{E5}D_{\rho_1}(x)(a,b)-D_{\rho_2}(x)(a,b)=[b,a,\varphi(x)]_{\mathfrak
h}.\end{equation}

 Denote the set of all non-abelian 3-cocycles on $\mathfrak g$ with values in
 $\mathfrak h$ by $Z_{nab}^{3}(\mathfrak g,\mathfrak h)$.
The equivalent class of a non-abelian 3-cocycle $(\omega,\theta,\rho)$ is denoted by $[(\omega,\theta,\rho)]$.
The quotient of $Z_{nab}^{3}(\mathfrak g,\mathfrak h)$ by the above equivalence
relation is denoted by $H_{nab}^{3}(\mathfrak g,\mathfrak h)$, which is called non-abelian cohomology
group of $\mathfrak g$ with coefficients in $\mathfrak h$.

Using the above notations, we define a trilinear map $[  \ , \ , \ ]_{(\omega,\theta,\rho)}$ on $\mathfrak g\oplus \mathfrak h$ by
\begin{align}\label{NLts}[x+a,y+b,z+c]_{(\omega,\theta,\rho)}=&[x,y,z]_{\mathfrak g}+\omega(x,y,z)+D_{\theta}(x,y)c+\theta(y,z)a-\theta(x,z)b
\nonumber\\&+D_{\rho}(z)(a,b)+\rho(x)(b,c)-\rho(y)(a,c)+[a,b,c]_{\mathfrak
h}\end{align}
for all $x,y,z\in \mathfrak g$ and $a,b,c\in \mathfrak h$.

\begin{pro} \label{LY} With the above notions,
$(\mathfrak g\oplus \mathfrak h,[  \  , \ , \
]_{(\omega,\theta,\rho)})$ is a Lie triple system if and only if the triple $(\omega,\theta,\rho)$
is a non-abelian 3-cocycle. Denote this Lie triple system
 $(\mathfrak g\oplus \mathfrak h,[  \  , \ , \
]_{(\omega,\theta,\rho)})$ simply by $\mathfrak g\oplus_{(\omega,\theta,\rho)}\mathfrak h$.
\end{pro}

\begin{proof}
$(\mathfrak g\oplus \mathfrak h,[  \  , \ , \
]_{(\omega,\theta,\rho)})$ is a Lie triple system if and only if
Eqs.~~(\ref{DLts1})-(\ref{DLts3}) hold for $[ \ , \ , \
]_{(\omega,\theta,\rho)}$. In fact, it is easy to check that Eqs.~~(\ref{DLts1})-(\ref{DLts2}) hold
for $[ \ , \ , \ ]_{(\omega,\theta,\rho)}$ if and only if (\ref{Lts02}), (\ref{Lts01})
hold.
 For the Eq.~~(\ref{DLts3}), we discuss it for the following cases:
 for all $x_1,x_2,y_1,y_2,y_3\in \mathfrak g\oplus \mathfrak h$,

(I) when all of the five elements $x_1,x_2,y_1,y_2,y_3$ belongs to
$\mathfrak g$, (\ref{DLts3}) holds if and only if (\ref{Lts1})
holds.

(II) when $x_1,x_2,y_1,y_2,y_3$ equal to:

(i) $x,y,z,w,a $ or $x,y,z,a,w $ or $x,a,y,z,w $
  respectively, (\ref{DLts3}) holds if and only if (\ref{Lts2}) holds.

(ii) $x,y,a,z,w$ or $a,x,y,z,w $ respectively, (\ref{DLts3}) holds
for if and only if (\ref{Lts3}) holds.

(iii) $x,y,z,a,b $ or $x,y,a,b,z $ or $x,y,a,z,b $ respectively,
(\ref{DLts3}) holds if and only if (\ref{Lts4}) holds.

(iv) $x,a,y,z,b $ or $a,x,y,z,b $ respectively, (\ref{DLts3}) holds
if and only if (\ref{Lts5}) holds.

(v) $x,a,y,b,z $ or $a,x,y,b,z $ or $x,a,b,y,z $ or $a,x,b,y,z $
respectively, (\ref{DLts3}) holds if and only if (\ref{Lts6}) holds.

(vi) $a,b,x,y,z $, (\ref{DLts3}) holds if and only if (\ref{Lts4}),
(\ref{Lts6}) hold.

(vii) $x,y,a,b,c $, (\ref{DLts3}) holds if and only if (\ref{Lts7})
 holds.

(viii) $x,a,y,b,c $ or $a,x,y,b,c $ or $x,a,b,c,y $ or $a,x,b,c,y $
or $a,x,b,y,c $ or $x,a,b,y,c $, (\ref{DLts3}) holds if and only if
(\ref{Lts8}) holds.

(ix) $a,b,x,c,y $ or $a,b,c,x,y $ or $a,b,x,y,c$, (\ref{DLts3})
holds if and only if (\ref{Lts9}) holds.

(x) $a,b,c,d,x $ or $a,b,c,x,d $ or $a,b,x,c,d$, (\ref{DLts3}) holds
if and only if (\ref{Lts10}) holds.

(xi) $a,x,b,c,d$ or $x,a,b,c,d $, (\ref{DLts3}) holds if and only if
(\ref{Lts11}) holds,

where $x,y,z,w\in \mathfrak g$ and $a,b,c,d\in \mathfrak h$ in all
the cases of (i)-(xi).

 This completes the proof.
\end{proof}

Let
 $\mathcal{E}:0\longrightarrow\mathfrak h\stackrel{i}{\longrightarrow} \hat{\mathfrak g}\stackrel{p}{\longrightarrow}\mathfrak g\longrightarrow0$
be a non-abelian extension of $\mathfrak g$ by
$\mathfrak h$ with a section $s$ of $p$.
Define $\omega_{s}:\mathfrak g\otimes \mathfrak g\otimes\mathfrak
g\rightarrow \mathfrak h,~\theta_{s}:\mathfrak g\wedge \mathfrak
g\rightarrow\mathfrak{gl}(\mathfrak h),~\rho_{s}:\mathfrak g\rightarrow
\mathrm{Hom} (\mathfrak h\wedge \mathfrak h,\mathfrak h)$
 respectively by
\begin{equation}\label{C1}\omega_{s}(x,y,z)=[s(x),s(y),s(z)]_{\hat{\mathfrak g}}-s[x,y,z]_{\mathfrak g},\end{equation}
\begin{equation}\label{C2}\theta_{s}(x,y)a=[a,s(x),s(y)]_{\hat{\mathfrak g}},~~~~\rho_{s}(x)(a,b)=[s(x),a,b]_{\hat{\mathfrak g}}.\end{equation}
By Eq.~~(\ref{C2}), we have
\begin{equation}\label{C3}D_{\theta_{s}}(x,y)a=[s(x),s(y),a]_{\hat{\mathfrak g}},~~~~
D_{\rho_{s}}(x)(a,b)=[a,b,s(x)]_{\hat{\mathfrak
g}}\end{equation}
for any $x,y,z\in \mathfrak g,a,b\in \mathfrak h$.

By direct computations, we have
\begin{pro} \label{CY} With the above notions, $(\omega_{s},\theta_{s},\rho_{s})$ is a
non-abelian 3-cocycle on $\mathfrak g$ with values in
 $\mathfrak h$. We call it the non-abelian 3-cocycle corresponding to the extension $\mathcal{E}$ induced by $s$.
 Naturally, $(\mathfrak g\oplus\mathfrak h, [ \ , \ , \ ]_{(\omega_{s},\theta_{s},\rho_{s})})$ is
 a Lie triple system. Denote this Lie triple system simply by $\mathfrak g\oplus_{(\omega_{s},\theta_{s},\rho_{s})}\mathfrak h$.
\end{pro}
In the following, we denote $(\omega_{s},\theta_{s},\rho_{s})$
by $(\omega,\theta,\rho)$ without ambiguity.

\begin{rmk} \label{Rk1} When $\mathcal{E}$ is an abelian extension of $\mathfrak g$ by
 $\mathfrak h$, the maps $\rho_s$ given by (\ref{C2}) becomes to zero. Then
 $(\mathfrak h,\theta_s)$ is a representation of $\mathfrak g$ \cite{35}.
Moreover, $\omega_s$ is a 3-cocycle of $\mathfrak g$ with coefficients in the representation $(\mathfrak h,\theta_s)$.
We call it the 3-cocycle corresponding to the abelian extension $\mathcal{E}$ induced by $s$. We always denote
$\omega_{s}$ simply by $\omega$ if there is no confusion.
 \end{rmk}

 \begin{lem} \label{Le1} Let
 $\mathcal{E}:0\longrightarrow\mathfrak h\stackrel{i}{\longrightarrow} \hat{\mathfrak g}\stackrel{p}{\longrightarrow}\mathfrak g\longrightarrow0$
be a non-abelian extension of $\mathfrak g$ by
$\mathfrak h$ with different sections $s_1,s_2$ of $p$. Assume that
 $(\omega_i,\theta_i,\rho_i)$ is the non-abelian 3-cocycle
 corresponding to the extension $\mathcal{E}$ induced by $s_i$~(i=1,2).
 Then $(\omega_1,\theta_1,\rho_1)$ and $(\omega_2,\theta_2,\rho_2)$
 are equivalent, that is, the equivalent classes of non-abelian 3-cocycles corresponding to
 a non-abelian extension
induced by a section are independent on the choice of sections.
 \end{lem}

\begin{proof}
Let $\hat{\mathfrak g}$ be a non-abelian extension
of $\mathfrak g$ by $\mathfrak h$. Assume that $s_1$ and $s_2$ are two
different sections of $p$, $(\omega_1,\theta_1,\rho_1)$ and
$(\omega_2,\theta_2,\rho_2)$ are the corresponding non-abelian
3-cocycles. Define a linear map $\varphi:\mathfrak
g\longrightarrow\mathfrak h$ by $\varphi(x)=s_2(x)-s_1(x)$. Since
$p\varphi(x)=ps_2(x)-ps_1(x)=0$, $\varphi$ is well defined. Thanks to
Eqs.~(\ref{C1})-(\ref{C3}), we get
\begin{eqnarray*}&&w_1(x,y,z)=[s_1(x),s_1(y),s_1(z)]_{\hat{\mathfrak g}}-s_{1}[x,y,z]_{\mathfrak g}\\&=&
[s_2(x)-\varphi(x),s_2(y)-\varphi(y),s_2(z)-\varphi(z)]_{\hat{\mathfrak
g}}-(s_{2}[x,y,z]_{\mathfrak g}-\varphi([x,y,z]_{\mathfrak g})
\\&=&[s_2(x),s_2(y),s_2(z)]_{\hat{\mathfrak g}}-[s_2(x),\varphi(y),s_2(z)]_{\hat{\mathfrak
g}}-[s_2(x),s_2(y),\varphi(z)]_{\hat{\mathfrak
g}}+[s_2(x),\varphi(y),\varphi(z)]_{\hat{\mathfrak
g}}\\&&-[\varphi(x),s_2(y),s_2(z)]_{\hat{\mathfrak
g}}+[\varphi(x),\varphi(y),s_2(z)]_{\hat{\mathfrak
g}}+[\varphi(x),s_2(y),\varphi(z)]_{\hat{\mathfrak
g}}-[\varphi(x),\varphi(y),\varphi(z)]_{\hat{\mathfrak
g}}\\&&-s_2[x,y,z]_{\mathfrak g}+\varphi([x,y,z]_{\mathfrak g})
\\&=&w_2(x,y,z)+\theta_2(x,z)\varphi(y)-D_{\theta_2}(x,y)\varphi(z)
+\rho_2(x)(\varphi(y),\varphi(z))-\theta_2(y,z)\varphi(x)+D_{\rho_2}(z)(\varphi(x),\varphi(y))
\\&&-\rho_2(y)(\varphi(x),\varphi(z)) -[\varphi(x),\varphi(y),\varphi(z)]_{\hat{\mathfrak
g}}+\varphi[x,y,z]_{\mathfrak
g},\end{eqnarray*} which yields that
 Eq.~(\ref{E2}) holds. Similarly, Eqs.~(\ref{E4}) and (\ref{E6})
 hold. This finishes the proof.
\end{proof}

According to Proposition \ref{LY} and Proposition \ref{CY}, given a non-abelian extension
 $\mathcal{E}:0\longrightarrow\mathfrak h\stackrel{i}{\longrightarrow} \hat{\mathfrak g}\stackrel{p}{\longrightarrow}\mathfrak g\longrightarrow0$
of $\mathfrak g$ by
$\mathfrak h$ with a section $s$ of $p$, we have a non-abelian 3-cocycle
 $(\omega_{s},\theta_{s},\rho_{s})$ and a Lie triple system $\mathfrak g\oplus_{(\omega_{s},\theta_{s},\rho_{s})} \mathfrak h$, which yields
that
$$\mathcal{E}_{(\omega_{s},\theta_{s},\rho_{s})}:0\longrightarrow\mathfrak h\stackrel{i}{\longrightarrow} \mathfrak g\oplus_{(\omega_{s},\theta_{s},\rho_{s})} \mathfrak h\stackrel{\pi}{\longrightarrow}\mathfrak g\longrightarrow0$$ is a non-abelian extension of $\mathfrak g$ by $\mathfrak h$. Since any element
$\hat{w}\in \hat{\mathfrak g}$ can be written as $\hat{w}=a+s(x)$ with $a\in \mathfrak h,x\in \mathfrak g$,
define a linear map
\begin{equation*} f:\hat{\mathfrak g}\longrightarrow \mathfrak g\oplus_{(\omega_{s},\theta_{s},\rho_{s})} \mathfrak h,~f(\hat{w})=f(a+s(x))=a+x.\end{equation*}
It is easy to check that $f$ is a homomorphism of Lie triple systems such that
the following commutative diagram holds:
 \begin{equation*} \xymatrix{
  \mathcal{E}:0 \ar[r] & \mathfrak h\ar@{=}[d] \ar[r]^-{i} & \hat{\mathfrak g}\ar[d]_-{f} \ar[r]^-{p} & \mathfrak g \ar@{=}[d] \ar[r] & 0\\
 \mathcal{E}_{(\omega_s)}:0 \ar[r] & \mathfrak h \ar[r]^-{i} & \mathfrak g\oplus_{(\omega_{s},\theta_{s},\rho_{s})} \mathfrak h \ar[r]^-{\pi} & \mathfrak g  \ar[r] & 0,}\end{equation*}
 which indicates that the non-abelian extensions $\mathcal{E}$ and $\mathcal{E}_{(\omega_{s},\theta_{s},\rho_{s})}$ of $\mathfrak g$ by
$\mathfrak h$ are equivalent. On the other hand, if $(\omega,\theta,\rho)$
 is a non-abelian 3-cocycle on $\mathfrak g$ with values in
 $\mathfrak h$, there is a Lie triple system $\mathfrak g\oplus_{(\omega,\theta,\rho)} \mathfrak h$, which yields the following
 non-abelian extension of $\mathfrak g$ by $\mathfrak h$:
 \begin{equation*}\mathcal{E}_{(\omega,\theta,\rho)}:0\longrightarrow\mathfrak h\stackrel{i}{\longrightarrow}\mathfrak g\oplus_{(\omega,\theta,\rho)} \mathfrak h
\stackrel{\pi}{\longrightarrow}\mathfrak g\longrightarrow0,\end{equation*}
where $i$ is the inclusion and $\pi$ is the projection.

In the sequel, we characterize the relationship between non-abelian cohomology groups and non-abelian extensions in the following.

\begin{pro}
 Let $\mathfrak g$ and $\mathfrak h$ be two Lie triple systems.
 Then the equivalent classes of non-abelian extensions of $\mathfrak g$ by $\mathfrak h$
are classified by the non-abelian cohomology group, that is,
 $\mathcal{E}_{nab}(\mathfrak g,\mathfrak h)\simeq H_{nab}^{3}(\mathfrak g,\mathfrak h)$.
 \end{pro}

 \begin{proof}
 Define a linear map
 \begin{equation*}\Theta:\mathcal{E}_{nab}(\mathfrak g,\mathfrak h)\rightarrow H_{nab}^{3}(\mathfrak g,\mathfrak h),~\end{equation*}
where $\Theta$ assigns an equivalent class of non-abelian extensions to the class of corresponding non-abelian 3-cocycles.
First, we check that $\Theta$ is well-defined.
 Assume that $\mathcal{E}_1$ and $\mathcal{E}_2$ are two equivalent non-abelian extensions of $\mathfrak g$ by $\mathfrak h$ via the map
 $f$, that is, the commutative diagram (\ref{Ene1}) holds. Let $s_1:\mathfrak g\rightarrow \hat{\mathfrak g}_1$ be a
section of $p_1$. Then $p_2fs_1=p_1s_1=I_{\mathfrak g}$, which
follows that $s_2=fs_1$ is a section of $p_2$. Let
 $(\omega_1,\theta_1,\rho_1)$ and
$(\omega_2,\theta_2,\rho_2)$ be two non-abelian 3-cocycles induced
by the sections $s_1,s_2$ respectively. Then we have,
\begin{eqnarray*}\theta_1(x,y)a&=&f(\theta_1(x,y)a)=f([a,s_1(x),s_1(y)]_{\hat{\mathfrak g}_1})
\\&=&[f(a),fs_1(x),fs_1(y)]_{\hat{\mathfrak g}_2}\\&=&[a,s_2(x),s_2(y)]_{\hat{\mathfrak
g}_2}\\&=&\theta_2(x,y)a .\end{eqnarray*} By the same token, we have
\begin{equation*}\omega_1(x,y,z)=\omega_2(x,y,z),~~\rho_1(x)(a,b)=\rho_2(x)(a,b).\end{equation*}
Thus,
$(\omega_1,\theta_1,\rho_1)=(\omega_2,\theta_2,\rho_2)$,
which means that $\Theta$ is well-defined.

 Next, we verify that $\Theta$ is injective. Indeed, suppose that
 $\Theta([\mathcal{E}_1])=[(\omega_1,\theta_1,\rho_1)]$ and $\Theta([\mathcal{E}_2])=[(\omega_2,\theta_2,\rho_2)]$. If
the equivalent classes $[(\omega_1,\theta_1,\rho_1)]=[(\omega_2,\theta_2,\rho_2)]$, we obtain that the non-abelian 3-cocycles
 $(\omega_1,\theta_1,\rho_1)$ and
$(\omega_2,\theta_2,\rho_2)$ are equivalent
via the linear map $\varphi:\mathfrak g\longrightarrow
\mathfrak h$, satisfying Eqs.~~(\ref{E2})-(\ref{E6}). Define a linear map
$f:\mathfrak g\oplus_{(\omega_1,\theta_1,\rho_1)} \mathfrak h\longrightarrow \mathfrak g\oplus_{(\omega_2,\theta_2,\rho_2)}
\mathfrak h$ by
\begin{equation*}f(x+a)=x-\varphi(x)+a,~~\forall~x\in \mathfrak g,a\in \mathfrak h.\end{equation*}
 According to Eq.~~(\ref{NLts}), for all $x,y,z\in \mathfrak
g,a,b,c\in \mathfrak h$, we get
\begin{eqnarray*}&&f([x+a,y+b,z+c]_{(\omega_1,\theta_1,\rho_1)})
\\&=&f([x,y,z]_{\mathfrak g}+\omega_1(x,y,z)+D_{\theta_1}(x,y)c+\theta_1(y,z)a-\theta_1(x,z)b
\\&&+D_{\rho_1}(z)(a,b)+\rho_1(x)(b,c)-\rho_1(y)(a,c)+[a,b,c]_{\mathfrak h})
\\&=&[x,y,z]_{\mathfrak g}-\varphi([x,y,z]_{\mathfrak g})+\omega_1(x,y,z)+D_{\theta_1}(x,y)c+\theta_1(y,z)a-\theta_1(x,z)b
\\&&+D_{\rho_1}(z)(a,b)+\rho_1(x)(b,c)-\rho_1(y)(a,c)+[a,b,c]_{\mathfrak h},\end{eqnarray*}
and
\begin{eqnarray*}&&[f(x+a),f(y+b),f(z+c)]_{(\omega_2,\theta_2,\rho_2)}
\\&=&[x-\varphi(x)+a,y-\varphi(y)+b,z-\varphi(z)+c]_{(\omega_2,\theta_2,\rho_2)}
\\&=&[x,y,z]_{\mathfrak g}+\omega_2(x,y,z)+D_{\theta_2}(x,y)(c-\varphi(z))+\theta_2(y,z)(a-\varphi(x))-\theta_2(x,z)(b-\varphi(y))
\\&&+T_{2}(z)(a-\varphi(x),b-\varphi(y))+\rho_2(x)(b-\varphi(y),c-\varphi(z))-\rho_2(y)(a-\varphi(x),c-\varphi(z))
\\&&+[a-\varphi(x),b-\varphi(y),c-\varphi(z)]_{\mathfrak h})
\\&=&[x,y,z]_{\mathfrak g}+\omega_2(x,y,z)+D_{\theta_2}(x,y)(c-\varphi(z))+\theta_2(y,z)(a-\varphi(x))-\theta_2(x,z)(b-\varphi(y))
\\&&+D_{\rho_2}(z)(\varphi(x),\varphi(y))-D_{\rho_2}(z)(a,\varphi(y))-D_{\rho_2}(z)(\varphi(x),b)+D_{\rho_2}(z)(a,b)
\\&&+\rho_2(x)(\varphi(y),\varphi(z))-\rho_2(x)(\varphi(y),c)-\rho_2(x)(b,\varphi(z))+\rho_2(x)(b,c)
-\rho_2(y)(\varphi(x),\varphi(z))\\&&+\rho_2(y)(\varphi(x),c)+\rho_2(y)(a,\varphi(z))-\rho_2(y)(a,c)
-[\varphi(x),\varphi(y),\varphi(z)]_{\mathfrak h}+[\varphi(x),\varphi(y),c]_{\mathfrak h}\\&&+[\varphi(x),b,\varphi(z)]_{\mathfrak h}
-[\varphi(x),b,c]_{\mathfrak h}+[a,\varphi(y),\varphi(z)]_{\mathfrak h}-[a,b,\varphi(z)]_{\mathfrak h}-[a,\varphi(y),c]_{\mathfrak h}
+[a,b,c]_{\mathfrak h}
.\end{eqnarray*}
In view of Eqs.~(\ref{E2})-(\ref{E5}), we have
\begin{equation*}f([x+a,y+b,z+c]_{(\omega_1,\theta_1,\rho_1)})=[f(x+a),f(y+b),f(z+c)]_{(\omega_2,\theta_2,\rho_2)}.\end{equation*}
Hence, $f$ is a homomorphism of Lie triple systems. Clearly, the
following commutative diagram holds:
\begin{equation}
\xymatrix{
 \mathcal{E}_{(\omega_1,\theta_1,\rho_1)}: 0 \ar[r] & \mathfrak h\ar@{=}[d] \ar[r]^-{i} & \mathfrak g\oplus_{(\omega_1,\theta_1,\rho_1)} \mathfrak h \ar[d]_-{f} \ar[r]^-{\pi} & \mathfrak g \ar@{=}[d] \ar[r] & 0\\
\mathcal{E}_{(\omega_2,\theta_2,\rho_2)}: 0 \ar[r] & \mathfrak h \ar[r]^-{i} & \mathfrak g\oplus_{(\omega_2,\theta_2,\rho_2)} \mathfrak h \ar[r]^-{\pi} & \mathfrak g  \ar[r] & 0
 .}\end{equation}
Thus $\mathcal{E}_{(\omega_1,\theta_1,\rho_1)}$ and $\mathcal{E}_{(\omega_2,\theta_2,\rho_2)}$ are equivalent
non-abelian extensions of $\mathfrak g$ by $\mathfrak h$,
which means that $[\mathcal{E}_{(\omega_1,\theta_1,\rho_1)}]=[\mathcal{E}_{(\omega_2,\theta_2,\rho_2)}]$. Thus, $\Theta$ is injective.

 Finally, we claim that $\Theta$ is surjective.
 For any equivalent class of non-abelian 3-cocycles $[(\omega)]$, by Proposition \ref{LY}, there is
 a non-abelian extension of $\mathfrak g$ by $\mathfrak h$:
   \begin{equation*}\mathcal{E}_{(\omega)}:0\longrightarrow\mathfrak h\stackrel{i}{\longrightarrow} \mathfrak g\oplus_{(\omega,\theta,\rho)} \mathfrak h\stackrel{\pi}{\longrightarrow}\mathfrak g\longrightarrow0.\end{equation*}
   Therefore, $\Theta([\mathcal{E}_{(\omega,\theta,\rho)}])=[(\omega,\theta,\rho)]$, which follows that $\Theta$ is surjective.
   In all, $\Theta$ is bijective. This finishes the proof.

 \end{proof}

\section{Non-abelian extensions in terms of Maurer Cartan elements}
 In this section, we classify the non-abelian extensions using Maurer Cartan
 elements. We start with recalling the Maurer-Cartan elements.

Let $(L=\oplus_{i}L_{i},[ \ , \ ], d)$ be a differential graded Lie
algebra. The set $\mathrm{MC}(L)$ of Maurer-Cartan elements of $(L,[ \ , \ ],
d)$ is defined by
$$\mathrm{MC}(L)=\{\eta\in L_1|d \eta+\frac{1}{2}[\eta,\eta]=0\}.$$
Moreover, $\eta_0,\eta_1\in \mathrm{MC}(L)$ are called gauge equivalent if
and only if there exists an element $\varphi\in L_{0}$ such that
$$\eta_1=e^{ad_{\varphi}}\eta_{0}-\frac{e^{ad_{\varphi}}-1}{ad_{\varphi}}d\varphi.$$

Let $\mathfrak g$ be a vector space. Consider the graded vector
space $C^{*}(\mathfrak g,\mathfrak g)=\oplus_{n\geq
0}C^{n}(\mathfrak g,\mathfrak g)$, where $C^{n}(\mathfrak
g,\mathfrak g)$ is the set of linear maps $f\in
\mathrm{Hom}(L\otimes\cdot\cdot\cdot\otimes L\otimes\mathfrak g,\mathfrak
g)$, where $L=\mathfrak g\otimes \mathfrak g$, satisfying
Eqs.~~(\ref{CO1}) and (\ref{CO2}) for all $X_{i}\in
L~(i=1,2,\cdot\cdot\cdot, n-1),x,y,z\in \mathfrak g$. The degree of
elements in $C^{n}(\mathfrak g,\mathfrak g)$ is $n$. Define
\begin{equation}\label{MC1}[f,g]_{Lts}=(-1)^{mn}i_{f}(g)-i_{g}(f),
~~\forall~f\in C^{n}(\mathfrak g,\mathfrak g),~g\in C^{m}(\mathfrak g,\mathfrak g),\end{equation}
where $i_{f}(g)\in C^{n+m}(\mathfrak g,\mathfrak g)$ is given by
\begin{equation}\label{MC2}i_{f}(g)=\sum_{k=1}^{n+1}(-1)^{m(k-1)}\sum_{\sigma\in sh(k-1,m)}(-1)^{\sigma}f\circ_{k}^{\sigma}g,\end{equation}
where $\sigma$ is a permutation in $(k-1,m)$-shuffle and
$f\circ_{k}^{\sigma}g$ is defined for $k=n+1$ by
\begin{equation}\label{MC3}f\circ_{n+1}^{\sigma}g(X_1,X_2,\cdot\cdot\cdot,X_{n+m},z)
=f(X_{\sigma(1)},\cdot\cdot\cdot,X_{\sigma(n)},g(X_{\sigma(n+1)},\cdot\cdot\cdot,X_{\sigma(n+m)},z))\end{equation}
and for $1\leq k\leq n$ by
\begin{align}\label{MC4}&f\circ_{k}^{\sigma}g(X_1,\cdot\cdot\cdot,X_{n+m},z)
\nonumber
\\=&f(X_{\sigma(1)},\cdot\cdot\cdot,X_{\sigma(k-1)},g(X_{\sigma(k)},\cdot\cdot\cdot,X_{\sigma(n+m-1)},x_{k+m})
,y_{k+m},X_{k+m+1},\cdot\cdot\cdot,X_{m+n},z)\nonumber\\&
+f(X_{\sigma(1)},\cdot\cdot\cdot,X_{\sigma(k-1)},x_{k+m},g(X_{\sigma(k)},\cdot\cdot\cdot,X_{\sigma(k+m-1)},y_{k+m}),X_{k+m+1},\cdot\cdot\cdot,X_{m+n},z)
\end{align}
for all $X_{i}\in L~(i=1,2,\cdot\cdot\cdot, n+m)$ and $z\in
\mathfrak g$.

\begin{pro} \cite{20} \label{pro:Dga1} With the above notations, $(C^{*}(\mathfrak g,\mathfrak g),[ \ , \
]_{Lts})$ is a graded Lie algebra. $\pi\in C^{1}(\mathfrak
g,\mathfrak g)=\mathrm{Hom}(\wedge^{2}\mathfrak g\otimes \mathfrak
g,\mathfrak g)$ defines a Lie triple system structure on $\mathfrak
g$ if and only if $[ \ , \ ]_{Lts}=0$, i.e, $\pi$ is a
Maurer-Cartan element of the graded Lie algebra $(C^{*}(\mathfrak
g,\mathfrak g),[ \ , \ ]_{Lts})$. Furthermore, $(C^{*}(\mathfrak
g,\mathfrak g),[ \ , \ ]_{Lts},d_{\pi})$ is a differential graded
Lie algebra, where $d_{\pi}$ is given by
\begin{equation}\label{MC5}d_{\pi}(f)=(-1)^{n-1}[\pi,f]_{Lts},~~\forall~f\in C^{n-1}(\mathfrak
g,\mathfrak g).\end{equation}
\end{pro}

Let $(\mathfrak g,\pi_{\mathfrak g})$
and $(\mathfrak h, \pi_{\mathfrak h})$ be two Lie triple systems. Then $(\mathfrak
g\oplus \mathfrak h,\pi_{\mathfrak g\oplus\mathfrak h})$ is a Lie triple system,
where $\pi_{\mathfrak g\oplus\mathfrak h}$ is defined by
\begin{equation*}\pi_{\mathfrak g\oplus\mathfrak h}(x + a, y +
b, z + c) =\pi_{\mathfrak g}(x, y, z) +\pi_{\mathfrak h}(a,b,c)\end{equation*}
 for all $x,y,z\in \mathfrak g,a,b,c\in \mathfrak h$.

In view of Proposition \ref{pro:Dga1}, $(C^{*}(\mathfrak g\oplus
\mathfrak h,\mathfrak g\oplus \mathfrak h),[ \ , \
]_{Lts},d_{\pi_{\mathfrak g\oplus \mathfrak h}})$ is a differential
graded Lie algebra.
Define $C_{>}^{n}(\mathfrak g\oplus \mathfrak
h,\mathfrak h)\subset C^{n}(\mathfrak g\oplus \mathfrak h,\mathfrak
h)$ by
$$C^{n}(\mathfrak g\oplus \mathfrak h,\mathfrak
h)=C_{>}^{n}(\mathfrak g\oplus \mathfrak h,\mathfrak h)\oplus C^{n}(
\mathfrak h,\mathfrak h).$$ Denote by $C_{>}(\mathfrak g\oplus
\mathfrak h,\mathfrak h)=\oplus_{n}C_{>}^{n}(\mathfrak g\oplus
\mathfrak h,\mathfrak h)$.

Similar to the case of $3$-Lie algebras \cite{016}, we have

\begin{pro} \label{pro:Dga2} With the above notations, $(C_{>}(\mathfrak g\oplus
\mathfrak h,\mathfrak h),[ \ , \ ]_{Lts},d_{\pi_{\mathfrak g\oplus \mathfrak h}})$ is a differential graded Lie subalgebra of
$(C^{*}(\mathfrak g\oplus \mathfrak h,\mathfrak g\oplus \mathfrak
h),[ \ , \ ]_{Lts},d_{\pi_{\mathfrak g\oplus \mathfrak h}})$.
Moreover, its degree $0$ part $C_{>}^{0}(\mathfrak g\oplus \mathfrak
h,\mathfrak h)=\mathrm{Hom} (\mathfrak h,\mathfrak h)$ is abelian.
\end{pro}

\begin{pro} The following conditions are equivalent:
\begin{enumerate}
    \item[(i)] $(\mathfrak g\oplus \mathfrak
h, [ \ , \ , \ ]_{(\omega,\theta,\rho)} )$ is a Lie triple system,
which is a non-abelian extension of $\mathfrak g$ by $\mathfrak h$;
    \item[(ii)] $\bar{\omega}$ is a Maurer-Cartan element of the differential graded Lie algebra $(C_{>}(\mathfrak g\oplus \mathfrak
h, \mathfrak h), [ \ ,  \ ]_{Lts},d_{\pi_{\mathfrak g\oplus \mathfrak h}})$, where
\begin{align*}\bar{\omega}(x+a,y+b,z+c)=&\omega(x,y,z)+D_{\theta}(x,y)c+\theta(y,z)a-\theta(x,z)b
\\&+D_{\rho}(z)(a,b)+\rho(x)(b,c)-\rho(y)(a,c),\end{align*}
for all $x,y,z\in \mathfrak g,a,b,c\in \mathfrak h$.
\end{enumerate}
\end{pro}

\begin{proof} According to the definition of Maurer-Cartan element,
$\bar{\omega}$ is a Maurer-Cartan element of the
differential graded Lie algebra $(C_{>}(\mathfrak g\oplus \mathfrak
h, \mathfrak h), [ \ ,  \ ]_{Lts},d_{\pi_{\mathfrak g\oplus \mathfrak h}})$ if and only if
$$d_{\pi_{\mathfrak g\oplus \mathfrak h}}\bar{\omega} +\frac{1}{2}[\bar{\omega},\bar{\omega}]_{Lts}=0.$$
Thus,
\begin{eqnarray*}&&d_{\pi_{\mathfrak g\oplus \mathfrak h}}\bar{\omega} +\frac{1}{2}[\bar{\omega},\bar{\omega}]_{Lts}
=[\pi_{\mathfrak g\oplus\mathfrak
h},\bar{\omega}]_{Lts}+\frac{1}{2}[\bar{\omega},\bar{\omega}]_{Lts}=\frac{1}{2}[\pi_{\mathfrak
g\oplus\mathfrak h}+\bar{\omega},\pi_{\mathfrak g\oplus\mathfrak
h}+\bar{\omega}]=0,\end{eqnarray*} which implies that $\pi_{\mathfrak
g\oplus\mathfrak h}+\bar{\omega}$ is a Maurer-Cartan element of the
differential graded Lie algebra $(C^{*}(\mathfrak g\oplus \mathfrak
h, \mathfrak g\oplus \mathfrak h), [ \ ,  \
]_{Lts},d_{\pi_{\mathfrak g\oplus \mathfrak h}})$. Combining
Proposition \ref{pro:Dga1}, $(\mathfrak g\oplus \mathfrak h, [ \ , \
, \ ]_{(\omega,\theta,\rho)} )$ is a Lie triple system and
$\bar{\omega}$ is a Maurer-Cartan element of $(C_{>}(\mathfrak
g\oplus \mathfrak h, \mathfrak h), [ \ ,  \
]_{Lts},d_{\pi_{\mathfrak g\oplus \mathfrak h}})$.
\end{proof}

\begin{pro}
Let $\mathfrak g$ and $\mathfrak h$ be two Lie triple systems. Then
the equivalent classes of non-abelian extensions $\mathfrak g$ by
$\mathfrak h$ are one-to-one correspond to the gauge equivalence classes
of Maurer-Cartan elements in the differential graded Lie algebra
$(C_{>}(\mathfrak g\oplus \mathfrak h, \mathfrak h), [ \ ,  \
]_{Lts},d_{\pi_{\mathfrak g\oplus \mathfrak h}})$.
\end{pro}

\begin{proof}
Let $\bar{\omega},\bar{\omega}'$ be
two Maurer-Cartan elements of the differential graded Lie algebra
$(C_{>}(\mathfrak g\oplus \mathfrak h, \mathfrak h), [ \ , \
]_{Lts},d_{\pi_{\mathfrak g\oplus \mathfrak h}})$, and $\bar{\omega},\bar{\omega}'$ are
gauge equivalent, there is a linear map $\varphi\in \mathrm{Hom} (\mathfrak
g, \mathfrak h)$ such that
\begin{eqnarray*}\bar{\omega}'&=&e^{ad_{\varphi}}\bar{\omega}-\frac{e^{ad_{\varphi}}-1}{ad_{\varphi}}d_{\pi_{\mathfrak g\oplus \mathfrak h}}\varphi
\\&=&(id+ad_{\varphi}+\frac{1}{2!}ad_{\varphi}^{2}+\cdot\cdot\cdot++\frac{1}{n!}ad_{\varphi}^{n}+\cdot\cdot\cdot)\bar{\omega}
\\&&-(id+ad_{\varphi}+\frac{1}{2!}ad_{\varphi}+\cdot\cdot\cdot+\frac{1}{n!}ad_{\varphi}^{n-1}+\cdot\cdot\cdot)d_{\pi_{\mathfrak g\oplus \mathfrak h}}\varphi.\end{eqnarray*}

For all $w_i=x_i+a_i\in \mathfrak g\oplus\mathfrak h~(i=1,2,3)$,
thanks to Eqs.~~(\ref{NLts}), (\ref{MC1})-(\ref{MC5}), we obtain
\begin{eqnarray*}ad_{\varphi}\bar{\omega}=[\varphi,\bar{\omega}]_{Lts}(w_1,w_2,w_3)&=&-\bar{\omega}(\varphi(w_1),w_2,w_3)-\bar{\omega}(w_1,\varphi(w_2),w_3)
-\bar{\omega}(w_1,w_2,\varphi(w_3))
\\&=&-\theta(x_2,x_3)\varphi(x_1)-D_{\rho}(x_3)(\varphi(x_1),a_2)+\rho(x_2)(\varphi(x_1),a_3)\\&&
+\theta(x_1,x_3)\varphi(x_2)-D_{\rho}(x_3)(a_1,\varphi(x_2))-\rho(x_1)(\varphi(x_2),a_3)\\&&-
D_{\theta}(x_1,x_2)\varphi(x_3)-\rho(x_1)(a_2,\varphi(x_3))+\rho(x_2)(a_1,\varphi(x_3)),
\end{eqnarray*}
\begin{eqnarray*}ad_{\varphi}^{2} \bar{\omega} &=&[\varphi,[\varphi,\bar{\omega}]_{Lts}]_{Lts}(w_1,w_2,w_3)
\\&=&-[\varphi,\bar{\omega}]_{Lts}(\varphi(w_1),w_2,w_3)-[\varphi,\bar{\omega}]_{Lts}(w_1,\varphi(w_2),w_3)-[\varphi,\bar{\omega}]_{Lts}(w_1,w_2,\varphi(w_3))
\\&=&D_{\rho}(x_3)(\varphi(x_1),\varphi(x_2))-\rho(x_2)(\varphi(x_1),\varphi(x_3))+D_{\rho}(x_3)(\varphi(x_1),\varphi(x_2))
\\&&+\rho(x_1)(\varphi(x_2),\varphi(x_3))-\rho(x_2)(\varphi(x_1),\varphi(x_3))+\rho(x_1)(\varphi(x_2),\varphi(x_3))
\\&=&2D_{\rho}(x_3)(\varphi(x_1),\varphi(x_2))-2\rho(x_2)(\varphi(x_1),\varphi(x_3))+2\rho(x_1)(\varphi(x_2),\varphi(x_3)),
\end{eqnarray*}
\begin{eqnarray*}d_{\pi_{\mathfrak
g\oplus\mathfrak h}}\varphi(w_1,w_2,w_3)&=&[\pi_{\mathfrak
g\oplus\mathfrak
h},\varphi]_{Lts}(w_1,w_2,w_3)\\&=&\pi_{\mathfrak
g\oplus\mathfrak h}(\varphi(w_1),w_2,w_3)+\pi_{\mathfrak
g\oplus\mathfrak h}(w_1,\varphi(w_2),w_3)\\&&+\pi_{\mathfrak
g\oplus\mathfrak h}(w_1,w_2,\varphi(w_3))-\varphi(\pi_{\mathfrak
g\oplus\mathfrak
h}(w_1,w_2,w_3)\\&=&[\varphi(x_1),a_2,a_3]_{\mathfrak
h}+[a_1,\varphi(x_2),a_3]_{\mathfrak
h}+[a_1,a_2,\varphi(x_3)]_{\mathfrak
h}-\varphi([x_1,x_2,x_3]_{\mathfrak g}),\end{eqnarray*}
\begin{eqnarray*}ad_{\varphi}(d_{\pi_{\mathfrak
g\oplus\mathfrak h}}\varphi)&=&[\varphi,d_{\pi_{\mathfrak
g\oplus\mathfrak
h}}\varphi]_{Lts}(w_1,w_2,w_3)\\&=&-d_{\pi_{\mathfrak
g\oplus\mathfrak h}}\varphi(\varphi(w_1),w_2,w_3)-d_{\pi_{\mathfrak
g\oplus\mathfrak h}}\varphi(w_1,\varphi(w_2),w_3)-d_{\pi_{\mathfrak
g\oplus\mathfrak h}}\varphi(w_1,w_2,\varphi(w_3))
\\&=&[\varphi(x_1),\varphi(x_2),a_3]_{\mathfrak
h}+[\varphi(x_1),a_2\varphi(x_3)]_{\mathfrak
h}+[\varphi(x_1),\varphi(x_2),a_3]_{\mathfrak
h}\\&&+[a_1,\varphi(x_2),\varphi(x_3)]_{\mathfrak h}
-[\varphi(x_1),a_2,\varphi(x_3)]_{\mathfrak
h}+[a_1,\varphi(x_2),\varphi(x_3)]_{\mathfrak
h}\\&=&2[a_1,\varphi(x_2),\varphi(x_3)]_{\mathfrak
h}+2[\varphi(x_1),a_2,\varphi(x_3)]_{\mathfrak
h}+2[\varphi(x_1),\varphi(x_2),a_3]_{\mathfrak h},\end{eqnarray*}
\begin{eqnarray*}ad_{\varphi}^{2}(d_{\pi_{\mathfrak
g\oplus\mathfrak h}}\varphi)&=&[\varphi,[\varphi,d_{\pi_{\mathfrak
g\oplus\mathfrak
h}}]_{Lts}]_{Lts}(w_1,w_2,w_3)\\&=&-[\varphi,d_{\pi_{\mathfrak
g\oplus\mathfrak
h}}]_{Lts}(\varphi(x_1),w_2,w_3)-[\varphi,d_{\pi_{\mathfrak
g\oplus\mathfrak h}}]_{Lts}(w_1,\varphi(x_2),w_3)
\\&&-[\varphi,d_{\pi_{\mathfrak g\oplus\mathfrak
h}}]_{Lts}(\varphi(x_1),w_2,\varphi(x_3))
\\&=&6[\varphi(x_1),\varphi(x_2),\varphi(x_3)]_{\mathfrak
h},\end{eqnarray*} and
$$ad_{\varphi}^{n}\bar{\omega}=0,~~ad_{\varphi}^{n}(d_{\pi_{\mathfrak
g\oplus\mathfrak h}}\bar{\omega})=0,~~\forall~~n\geq 3.$$
 Therefore, Maurer-Cartan elements $\bar{\omega}'$ and $\bar{\omega}$ are gauge equivalent if and only if
\begin{eqnarray*}&&\bar{\omega}'=\bar{\omega}+[\varphi,\pi]_{Lts}+\frac{1}{2!}[\varphi,[\varphi,\pi]_{Lts}]_{Lts}\\&&-(d_{\pi_{\mathfrak
g\oplus\mathfrak h}}\varphi +\frac{1}{2!}[\varphi,d_{\pi_{\mathfrak
g\oplus\mathfrak
h}}\varphi]_{Lts}+\frac{1}{3!}[\varphi,[\varphi,d_{\pi_{\mathfrak
g\oplus\mathfrak h}}\varphi]_{Lts}]_{Lts}).
\end{eqnarray*}
Hence, two Maurer-Cartan elements $\bar{\omega}$ and $\bar{\omega}'$ are equivalent if and
only if Eqs.~~(\ref{E2})-(\ref{E6}) hold. The proof is completed.

\end{proof}

\section{Inducibility of pairs of  Lie triple system automorphisms}
In this section, we study inducibility of a pair of Lie triple system
automorphisms and characterize them by equivalent conditions.

Let $\mathfrak g$ and $\mathfrak h$ be two Lie triple systems, and
  $$0\longrightarrow\mathfrak h\stackrel{i}{\longrightarrow} \hat{\mathfrak g}\stackrel{p}{\longrightarrow}\mathfrak g\longrightarrow0,$$
 be a non-abelian
extension of $\mathfrak g$ by $\mathfrak h$ with a section $s$ of
$p$. For any automorphism
$\gamma\in \mathrm{Aut} (\hat{\mathfrak g})$, then $\gamma\in \mathrm{Aut} (\mathfrak
h)$ if and only if $\gamma (\mathfrak h)\subseteq \mathfrak h$.
Denote $\mathrm{Aut}_{\mathfrak h}
(\hat{\mathfrak g})=\{\gamma\in \mathrm{Aut} (\hat{\mathfrak g})\mid \gamma
(\mathfrak h)=\mathfrak h\}.$ We define a linear map $\bar{\gamma}:\mathfrak
g\longrightarrow \mathfrak g$ by $$\bar{\gamma}(x)=p\gamma
s(x),~~\forall~x\in \mathfrak g.$$

Assume that $s_1$ and $s_2$ are two distinct sections of
$p$, since $ps_1(x)-ps_2(x)=0$, $s_1(x)-s_2(x)\in
\mathrm{Ker}p\cong \mathfrak h$, which indicates that $\gamma (s_1(x)-s_2(x))\in
\mathfrak h$. Thus, $p\gamma s_1(x)=p\gamma s_2(x)$, which yields
that $\bar{\gamma}$ is independent on the choice of sections.

For all $x,y,z\in \mathfrak g$, since $p|_{\mathfrak h}=0$,
\begin{eqnarray*}&&\bar{\gamma}([x,y,z]_{\mathfrak g})=p\gamma (s[x,y,z]_{\mathfrak g})
\\&=&p\gamma ([s(x),s(y),s(z)]_{\hat{\mathfrak g}}-\omega(x,y,z))
\\&=&p \gamma ([s(x),s(y),s(z)]_{\hat{\mathfrak g}})
\\&=& [p\gamma s(x),p\gamma s(y),p\gamma s(z)]_{\mathfrak g}
\\&=&[\bar{\gamma}(x),\bar{\gamma}(y),\bar{\gamma}(z)]_{\mathfrak g}),\end{eqnarray*}
which yields that $\bar{\gamma}$ is a homomorphism of Lie triple systems. It
is easy to check that $\bar{\gamma}$ is bijective. Thus, $\bar{\gamma}\in
\mathrm{Aut} (\mathfrak g)$. Then we can define a group homomorphism
$$\lambda:\mathrm{Aut}_{\mathfrak h} (\hat{\mathfrak g})\longrightarrow \mathrm{Aut}
(\mathfrak g)\times \mathrm{Aut} (\mathfrak
h),~~\lambda(\gamma)=(\bar{\gamma},\gamma|_{\mathfrak h}).$$

\begin{defi}
 A pair $(\alpha,\beta)\in \mathrm{Aut} (\mathfrak g)\times \mathrm{Aut} (\mathfrak
 h)$ is said to be inducible if $(\alpha,\beta)$ is an image of
 $\lambda$.
\end{defi}

In the following, we investigate when a pair $(\alpha,\beta)$ is
inducible.

\begin{thm} \label{EC} Let $0\longrightarrow\mathfrak h\stackrel{i}{\longrightarrow}
\hat{\mathfrak g}\stackrel{p}{\longrightarrow}\mathfrak
g\longrightarrow0$ be a non-abelian extension of the Lie triple system
$\mathfrak g$ by $\mathfrak h$ with a section $s$ of $p$ and
$(\omega,\theta,\rho)$ be the corresponding non-abelian 3-cocycle
induced by $s$. A pair $(\alpha,\beta)\in \mathrm{Aut}(\mathfrak
g)\times \mathrm{Aut}(\mathfrak h)$ is inducible if and only if there is a
linear map $\varphi:\mathfrak g\longrightarrow \mathfrak h$
satisfying the following conditions:
\begin{align}\label{Iam1}
     \beta(\theta(x,y)a)-\theta(\alpha(x),\alpha(y))\beta(a)=[\beta(a),\varphi(x),\varphi(y)]_{\mathfrak
     h}-D_{\rho}(\alpha(y))(\beta(a),\varphi(x))+\rho(\alpha(x))(\beta(a),\varphi(y)),
\end{align}
\begin{equation}\label{Iam2}
     \beta(\rho(x)(a,b))-\rho(\alpha(x))(\beta(a),\beta(b))=[\beta(a),\varphi(x),\beta(b)]_{\mathfrak
     h},
\end{equation}
\begin{align}\label{Iam3}
     &\beta\omega(x,y,z)-\omega(\alpha(x),\alpha(y),\alpha(z))=D_{\rho}(\alpha(z))(\varphi(x),\varphi(y))-\rho(\alpha(y))(\varphi(x),\varphi(z))-
\theta(\alpha(y),\alpha(z))\varphi(x)\\
&+\rho(\alpha(x))(\varphi(y),\varphi(z))+\theta(\alpha(x),\alpha(z))\varphi(y)
-D_{\theta}(\alpha(x),\alpha(y))\varphi(z)+\varphi([x,y,z]_{\mathfrak
     g})-[\varphi(x),\varphi(y),\varphi(z)]_{\mathfrak
     h},\nonumber
\end{align}
for all $x,y,z\in \mathfrak g$ and $a,b\in \mathfrak h$.
\end{thm}

\begin{proof} Suppose that $(\alpha,\beta)\in \mathrm{Aut}(\mathfrak
g)\times \mathrm{Aut}(\mathfrak h)$ is inducible, then there is an
automorphism $\gamma\in \mathrm{Aut}_{\mathfrak h}(\hat{\mathfrak g})$ such
that $\gamma|_{\mathfrak h}=\beta$ and $p\gamma s=\alpha$. For all
$x\in \mathfrak g$, since $s$ is a section of $p$,
that is, $ps=I_{\mathfrak g}$,
$$p(\gamma s-s\alpha)(x)=\alpha(x)-\alpha(x)=0,$$
which implies that $(\gamma s-s\alpha)(x)\in \mathrm{ker}p\cong \mathfrak h$.
So we can define a linear map $\varphi:\mathfrak g\longrightarrow
\mathfrak h$ by
$$\varphi(x)=(s\alpha-\gamma s)(x),~~\forall~x\in \mathfrak g.$$
For $x,y\in \mathfrak g, a\in \mathfrak h$, using
Eqs.~~(\ref{C2})-(\ref{C3}), we get
\begin{eqnarray*}&&
      \beta(\theta(x,y)a)-\theta(\alpha(x),\alpha(y))\beta(a)\\&=&\beta[a,s(x),s(y)]_{\hat{\mathfrak g}}-[\beta(a),s\alpha(x),s\alpha(y)]_{\hat{\mathfrak g}}
      \\&=&[\beta(a),\beta s(x),\beta s(y)]_{\hat{\mathfrak g}}-[\beta(a),s\alpha(x),s\alpha(y)]_{\hat{\mathfrak g}}
      \\&=&[\beta(a),\beta s(x)-s\alpha(x),\beta s(y)]_{\hat{\mathfrak
      g}}+[\beta(a),s\alpha(x),\beta s(y)]_{\hat{\mathfrak g}}-[\beta(a),s\alpha(x),s\alpha(y)]_{\hat{\mathfrak g}}
      \\&=&[\beta(a),\varphi(x),\beta s(y)]_{\hat{\mathfrak
      g}}+[\beta(a),s\alpha(x),\varphi(y)]_{\hat{\mathfrak g}}
\\&=&[\beta(a),\varphi(x),\beta s(y)-s\alpha(y)]_{\hat{\mathfrak
      g}}+[\beta(a),\varphi(x),s\alpha(y)]_{\hat{\mathfrak
      g}}+ [\beta(a),s\alpha(x),\varphi(y)]_{\hat{\mathfrak g}}
\\&=&[\beta(a),\varphi(x),\varphi(y)]_{\hat{\mathfrak
      g}}+[\beta(a),\varphi(x),s\alpha(y)]_{\hat{\mathfrak
      g}}- [s\alpha(x),\beta(a),\varphi(y)]_{\hat{\mathfrak g}}
\\&=&[\beta(a),\varphi(x),\varphi(y)]_{\mathfrak
     h}+D_{\rho}(\alpha(y))(\beta(a),\varphi(x))-\rho(\alpha(x))(\beta(a),\varphi(y)),
\end{eqnarray*}
that is, Eq.~~(\ref{Iam1}) holds. Similarly, Eqs.~~(\ref{Iam2}) and
(\ref{Iam3}) hold.

Conversely, suppose that $(\alpha,\beta)\in \mathrm{Aut}(\mathfrak g)\times
\mathrm{Aut}(\mathfrak h)$ and there is a linear map $\varphi:\mathfrak
g\longrightarrow \mathfrak h$ satisfying Eqs.~~(\ref{Iam1})-
(\ref{Iam3}). Due to $s$ being a section of $p$,
then all $\hat{w}\in \hat{\mathfrak g}$ can be written as
$\hat{w}=a+s(x)$ for some $a\in \mathfrak h,x\in \mathfrak g.$
Define a linear map $\gamma:\hat{\mathfrak g}\longrightarrow
\hat{\mathfrak g}$ by
$$\gamma(\hat{w})=\gamma(a+s(x))=\beta(a)-\varphi(x)+s\alpha(x).$$
If $\gamma(\hat{w})=0,$ then $s\alpha(x)=0$ and
$\beta(a)+\varphi(x)=0$. In view of $s$ and $\alpha$ being
injective, we get $x=0$, it follows that $a=0$. Thus,
$\hat{w}=a+s(x)=0$, that is $\gamma$ is injective. For any
$\hat{w}=a+s(x)\in \hat{\mathfrak g}$,
$$\gamma(\beta^{-1}(a)+\beta^{-1}\varphi\alpha^{-1}(x)+s\alpha^{-1}(x))=a+s(x)=\hat{w},$$
which yields that $\gamma$ is surjective. In all, $\gamma$ is
bijective.

Next we show that $\gamma$ is a homomorphism of Lie triple systems.
In fact, for all $\hat{w}_i=a_i+s(x_i)\in \hat{\mathfrak
g}~(i=1,2,3)$,
\begin{eqnarray*}&&
      [\gamma(\hat{w}_1),\gamma(\hat{w}_2),\gamma(\hat{w}_3)]_{\hat{\mathfrak g}}
\\&=&[\beta(a_1)-\varphi(x_1)+s\alpha(x_1),\beta(a_2)-\varphi(x_2)+s\alpha(x_2),\beta(a_3)-\varphi(x_3)+s\alpha(x_3)]_{\hat{\mathfrak g}}
\\&=&[\beta(a_1),\beta(a_2),\beta(a_3)]_{\hat{\mathfrak g}}-[\beta(a_1),\beta(a_2),\varphi(x_3)]_{\hat{\mathfrak
g}}+[\beta(a_1),\beta(a_2),s\alpha(x_3)]_{\hat{\mathfrak g}}
\\&&-[\beta(a_1),\varphi(x_2),\beta(a_3)]_{\hat{\mathfrak
g}}+[\beta(a_1),\varphi(x_2),\varphi(x_3)]_{\hat{\mathfrak
g}}-[\beta(a_1),\varphi(x_2),s\alpha(x_3)]_{\hat{\mathfrak g}}
\\&&+[\beta(a_1),s\alpha(x_2),\beta(a_3)]_{\hat{\mathfrak
g}}-[\beta(a_1),s\alpha(x_2),\varphi(x_3)]_{\hat{\mathfrak
g}}+[\beta(a_1),s\alpha(x_2),s\alpha(x_3)]_{\hat{\mathfrak g}}
\\&&-[\varphi(x_1),\beta(a_2),\beta(a_3)]_{\hat{\mathfrak
g}}+[\varphi(x_1),\beta(a_2),\varphi(x_3)]_{\hat{\mathfrak
g}}-[\varphi(x_1),\beta(a_2),s\alpha(x_3)]_{\hat{\mathfrak g}}
\\&&+[\varphi(x_1),\varphi(x_2),\beta(a_3)]_{\hat{\mathfrak
g}}-[\varphi(x_1),\varphi(x_2),\varphi(x_3)]_{\hat{\mathfrak
g}}+[\varphi(x_1),\varphi(x_2),s\alpha(x_3)]_{\hat{\mathfrak g}}
\\&&-[\varphi(x_1),s\alpha(x_2),\beta(a_3)]_{\hat{\mathfrak
g}}+[\varphi(x_1),s\alpha(x_2),\varphi(x_3)]_{\hat{\mathfrak
g}}-[\varphi(x_1),s\alpha(x_2),s\alpha(x_3)]_{\hat{\mathfrak g}}
\\&&+[s\alpha(x_1),\beta(a_2),\beta(a_3)]_{\hat{\mathfrak
g}}-[s\alpha(x_1),\beta(a_2),\varphi(x_3)]_{\hat{\mathfrak
g}}+[s\alpha(x_1),\beta(a_2),s\alpha(x_3)]_{\hat{\mathfrak g}}
\\&&-[s\alpha(x_1),\varphi(x_2),\beta(a_3)]_{\hat{\mathfrak
g}}+[s\alpha(x_1),\varphi(x_2),\varphi(x_3)]_{\hat{\mathfrak
g}}-[s\alpha(x_1),\varphi(x_2),s\alpha(x_3)]_{\hat{\mathfrak g}}
\\&&+[s\alpha(x_1),s\alpha(x_2),\beta(a_3)]_{\hat{\mathfrak
g}}-[s\alpha(x_1),s\alpha(x_2),\varphi(x_3)]_{\hat{\mathfrak
g}}+[s\alpha(x_1),s\alpha(x_2),s\alpha(x_3)]_{\hat{\mathfrak g}}
\\&=&\beta([a_1,a_2,a_3]_{\hat{\mathfrak
g}})-[\beta(a_1),\beta(a_2),\varphi(x_3)]_{\hat{\mathfrak
g}}+D_{\rho}(\alpha(x_3))(\beta(a_1),\beta(a_2))
\\&&-[\beta(a_1),\varphi(x_2),\beta(a_3)]_{\hat{\mathfrak
g}}+[\beta(a_1),\varphi(x_2),\varphi(x_3)]_{\hat{\mathfrak
g}}-D_{\rho}(\alpha(x_3))(\beta(a_1),\varphi(x_2))
\\&&-\rho(\alpha(x_2))(\beta(a_1),\beta(a_3))+\rho(\alpha(x_2))
(\beta(a_1),\varphi(x_3))+\theta(\alpha(x_2),\alpha(x_3))\beta(a_1)
\\&&-[\varphi(x_1),\beta(a_2),\beta(a_3)]_{\hat{\mathfrak
g}}+[\varphi(x_1),\beta(a_2),\varphi(x_3)]_{\hat{\mathfrak
g}}-D_{\rho}(\alpha(x_3))(\varphi(x_1),\beta(a_2))
\\&&+[\varphi(x_1),\varphi(x_2),\beta(a_3)]_{\hat{\mathfrak
g}}-[\varphi(x_1),\varphi(x_2),\varphi(x_3)]_{\hat{\mathfrak
g}}+D_{\rho}(\alpha(x_3))(\varphi(x_1),\varphi(x_2))
\\&&+\rho(\alpha(x_2))(\varphi(x_1),\beta(a_3))-\rho(\alpha(x_2))(\varphi(x_1),\varphi(x_3))
-\theta(\alpha(x_2),\alpha(x_3))\varphi(x_1)
\\&&+\rho(\alpha(x_1))(\beta(a_2),\beta(a_3))-\rho(\alpha(x_1))(\beta(a_2),\varphi(x_3))
-\theta(\alpha(x_1),\alpha(x_3))\beta(a_2)
\\&&-\rho(\alpha(x_1))(\varphi(x_2),\beta(a_3))+\rho(\alpha(x_1))(\varphi(x_2),\varphi(x_3))
+\theta(\alpha(x_1),\alpha(x_3))\varphi(x_2)
\\&&+D_{\theta}(\alpha(x_1),\alpha(x_2))\beta(a_3)-D_{\theta}(\alpha(x_1),\alpha(x_2))\varphi(x_3)
+\omega(\alpha(x_1),\alpha(x_2),\alpha(x_3))\\&&+s[\alpha(x_1),\alpha(x_2),\alpha(x_3)]_{\mathfrak g}
\end{eqnarray*}
and
\begin{eqnarray*}&&
      \gamma([\hat{w}_1,\hat{w}_2,\hat{w}_3]_{\hat{\mathfrak g}})
\\&=&\gamma([a_1,a_2,a_3]_{\hat{\mathfrak g}}+[a_1,a_2,s(x_3)]_{\hat{\mathfrak g}}+[a_1,s(x_2),a_3]_{\hat{\mathfrak g}}
+[a_1,s(x_2),s(x_3)]_{\hat{\mathfrak
g}}+[s(x_1),a_2,s(x_3)]_{\hat{\mathfrak
g}}\\&&+[s(x_1),s(x_2),a_3]_{\hat{\mathfrak
g}}+[s(x_1),a_2,a_3]_{\hat{\mathfrak g}}
+\omega(x_1,x_2,x_3)+s[x_1,x_2,x_3]_{\mathfrak g})
\\&=&\beta([a_1,a_2,a_3]_{\hat{\mathfrak g}})+\beta(D_{\rho}(x_3)(a_1,a_2)-\rho(x_2)
(a_1,a_3) +\theta(x_2,x_3)a_1
-\theta(x_1,x_3)a_2
\\&&+D_{\theta}(x_1,x_2)a_3+\rho(x_1)(a_2,a_3))+\beta\omega(x_1,x_2,x_3)-\varphi[x_1,x_2,x_3]_{\mathfrak g}+s\alpha[x_1,x_2,x_3]_{\mathfrak g}.
\end{eqnarray*}
Thanks to Eqs.~~(\ref{Iam1})-(\ref{Iam3}), we have
$$ \gamma[\hat{w}_1,\hat{w}_2,\hat{w}_3]_{\hat{\mathfrak
g}}=[\gamma(\hat{w}_1),\gamma(\hat{w}_2),\gamma(\hat{w}_3)]_{\hat{\mathfrak
g}},$$ which implies that $\gamma$ is a homomorphism of Lie triple
systems. Thus, $\gamma\in \mathrm{Aut}(\hat{\mathfrak g})$. Finally, we show
that $\gamma|_{\mathfrak h}=\beta$ and $p\gamma s=\alpha$. In fact,
$$\gamma(a)=\gamma(a+s(0))=\beta(a),~~\forall~a\in \mathfrak h$$ and
$$(p\gamma s)(x)=p\gamma(0+s(x))=p(s\alpha(x)-\varphi(x))=ps\alpha(x)=\alpha(x),~~\forall~x\in \mathfrak g.$$
Therefore, $\gamma|_{\mathfrak h}=\beta$ and $p\gamma s=\alpha$.
Thus, $(\alpha,\beta)\in \mathrm{Aut}(\mathfrak g)\times \mathrm{Aut}(\mathfrak h)$ is
inducible.
\end{proof}

Let $0\longrightarrow\mathfrak h\stackrel{i}{\longrightarrow}
\hat{\mathfrak g}\stackrel{p}{\longrightarrow}\mathfrak
g\longrightarrow0$ be a non-abelian extension of the Lie triple system
$\mathfrak g$ by $\mathfrak h$ with a section $s$ of $p$ and $(\omega,\theta,\rho)$ be the
corresponding non-abelian 3-cocycle induced by $s$.

 For all $(\alpha,\beta)\in \mathrm{Aut}(\mathfrak
g)\times \mathrm{Aut}(\mathfrak h)$. Define maps $\omega_{(\alpha,\beta)}:\mathfrak g\times
\mathfrak g\times\mathfrak g\longrightarrow \mathfrak
h,~\theta_{(\alpha,\beta)}:\mathfrak g\times \mathfrak
g\longrightarrow \mathfrak{gl}(\mathfrak h),~\rho_{(\alpha,\beta)}:\mathfrak
g\longrightarrow \mathrm{Hom} (\mathfrak h\times \mathfrak h,\mathfrak h)$
 respectively by
 \begin{equation}\label{Inc1}\omega_{(\alpha,\beta)}(x,y,z)=\beta\omega(\alpha^{-1}(x),\alpha^{-1}(y),\alpha^{-1}(z)),\end{equation}
 \begin{equation}\label{Inc2}\theta_{(\alpha,\beta)}(x,y)a=\beta\theta(\alpha^{-1}(x),\alpha^{-1}(y))\beta^{-1}(a)
,\end{equation}
 \begin{equation}\label{Inc3}\rho_{(\alpha,\beta)}(x)(a,b)=\beta \rho(\alpha^{-1}(x))(\beta^{-1}(a),\beta^{-1}(b)).\end{equation}
 By direct computations, we get
\begin{equation}\label{Inc4}D_{\theta_{(\alpha,\beta)}}(x,y)=\beta D_{\theta}(\alpha^{-1}(x),\alpha^{-1}(y))\beta^{-1}(a),
~~D_{\rho_{(\alpha,\beta)}}(a,b)=\beta D_{\rho}(\alpha^{-1}(x))(\beta^{-1}(a),\beta^{-1}(b)).\end{equation}
for all $x,y,z\in \mathfrak g,a,b\in \mathfrak h$.
\begin{pro} With the above notations,
$(\omega_{(\alpha,\beta)},\theta_{(\alpha,\beta)},\rho_{(\alpha,\beta)})$
is a non-abelian 3-cocycle. \end{pro}

\begin{proof} For all $x_1,x_2,y_1,y_2,y_3\in \mathfrak g$,
using Eqs.~(\ref{Lts1}), (\ref{Inc1}), (\ref{Inc2}) and
(\ref{Inc4}), we get
\begin{eqnarray*}&&D_{\theta_{(\alpha,\beta)}}(x_1,x_2)\omega_{(\alpha,\beta)}(y_1,y_2,y_3)+\omega_{(\alpha,\beta)}(x_1,x_2,[y_1,y_2,y_3]_{\mathfrak
g})\\&=&\beta\Big(
D_{\theta}(\alpha^{-1}(x_1),\alpha^{-1}(x_2))\beta^{-1}\big(\beta\omega_{(\alpha,\beta)}(\alpha^{-1}(y_1),\alpha^{-1}(y_2),\alpha^{-1}(y_3))\big)\Big)
\\&&-\beta\Big(\omega_{(\alpha,\beta)}(\alpha^{-1}(x_1),\alpha^{-1}(x_2),\alpha^{-1}([y_1,y_2,y_3]_{\mathfrak
g}))\Big)
\\&=&
\beta
\omega([\alpha^{-1}(x_1),\alpha^{-1}(x_2),\alpha^{-1}(y_1)]_{\mathfrak
g},\alpha^{-1}(y_2),\alpha^{-1}(y_3))\\&&+\beta\Big(\theta(\alpha^{-1}(y_2),\alpha^{-1}(y_3))\omega(\alpha^{-1}(x_1),\alpha^{-1}(x_2),\alpha^{-1}(y_1))\Big)
\\&&+\beta\omega(\alpha^{-1}(y_1),[\alpha^{-1}(x_1),\alpha^{-1}(x_2),\alpha^{-1}(y_2)]_{\mathfrak
g},\alpha^{-1}(y_3))\\&&-\beta\Big(\theta(\alpha^{-1}(y_1),\alpha^{-1}(y_3))\omega(\alpha^{-1}(x_1),\alpha^{-1}(x_2),\alpha^{-1}(y_2))\Big)
\\&&+\beta\omega(\alpha^{-1}(y_1),\alpha^{-1}(y_2),[\alpha^{-1}(x_1),\alpha^{-1}(x_2),\alpha^{-1}(y_3)]_{\mathfrak
g})\\&&+\beta\Big(
D_{\theta}(\alpha^{-1}(y_1),\alpha^{-1}(y_2))\omega(\alpha^{-1}(x_1),\alpha^{-1}(x_2),\alpha^{-1}(y_3))\Big)
\\&=&\omega_{(\alpha,\beta)}([x_1,x_2,y_1]_{\mathfrak
g},y_2,y_3)+\theta_{(\alpha,\beta)}(y_2,y_3)\omega_{(\alpha,\beta)}(x_1,x_2,y_1)+\omega_{(\alpha,\beta)}(y_1,[x_1,x_2,y_2]_{\mathfrak
g},y_3)\\&&-\theta_{(\alpha,\beta)}(y_1,y_3)\omega_{(\alpha,\beta)}(x_1,x_2,y_2)+\omega_{(\alpha,\beta)}(y_1,y_2,[x_1,x_2,y_3]_{\mathfrak
g})+D_{\theta_{(\alpha,\beta)}}(y_1,y_2)\omega_{(\alpha,\beta)}(x_1,x_2,y_3),
 \end{eqnarray*}
which implies that Eq.~~(\ref{Lts1}) holds for
$\omega_{(\alpha,\beta)},\theta_{(\alpha,\beta)}$. Take the same procedure,
Eqs.~~(\ref{Lts02}) and (\ref{Lts2})-(\ref{Lts11}) hold for
$\omega_{(\alpha,\beta)},\theta_{(\alpha,\beta)},D_{\theta_{(\alpha,\beta)}},\rho_{(\alpha,\beta)}$ and
$D_{\rho_{(\alpha,\beta)}}$. The proof is finished.
\end{proof}

\begin{thm} \label{Eth1} Let $0\longrightarrow\mathfrak h\stackrel{i}{\longrightarrow}
\hat{\mathfrak g}\stackrel{p}{\longrightarrow}\mathfrak
g\longrightarrow0$ be a non-abelian extension of
$\mathfrak g$ by $\mathfrak h$ with a section $s$ of $p$ and $(\omega,\theta,\rho)$ be the
corresponding non-abelian 3-cocycle induced by $s$. A pair $(\alpha,\beta)\in \mathrm{Aut}(\mathfrak
g)\times \mathrm{Aut}(\mathfrak h)$ is inducible if and only if the
non-abelian 3-cocycles $(\omega,\theta,\rho)$ and
$(\omega_{(\alpha,\beta)},\theta_{(\alpha,\beta)},\rho_{(\alpha,\beta)})$
are equivalent.
\end{thm}

\begin{proof}
Suppose $(\alpha,\beta)\in \mathrm{Aut}(\mathfrak g)\times \mathrm{Aut}(\mathfrak h)$
is inducible, then by Theorem ~\ref{EC}, there is a linear map
$\varphi:\mathfrak g\longrightarrow \mathfrak h$ satisfying
Eqs.~~(\ref{Iam1})-(\ref{Iam3}). For all $x,y\in \mathfrak g,a\in
\mathfrak h$, there exist $x_0,y_0\in \mathfrak g,a_0\in \mathfrak
h$ such that $x=\alpha(x_0),y=\alpha(y_0),a=\beta(a_0)$. Therefore,
using Eqs.~~(\ref{Iam1}) and (\ref{Inc2}),
\begin{eqnarray*}&&
\theta_{\alpha,\beta}(x,y)a-\theta(x,y)a\\&=&\beta(\theta(\alpha^{-1}(x),\alpha^{-1}(y))\beta^{-1}(a))
-\theta(x,y)a
\\&=& \beta(\theta(x_0,y_0)a_0)-\theta(\alpha(x_0),\alpha(y_0))\beta(a_0)
\\&=&[\beta(a_0),\varphi(x_0),\varphi(y_0)]_{\mathfrak
h}-D_{\rho}(\alpha(y_0))(\beta(a_0),\varphi(x_0))+\rho(x_0)(\beta(a_0),\varphi(y_0))
\\&=&[a,\varphi\alpha^{-1}(x),\varphi\alpha^{-1}(y)]_{\mathfrak
h}-D_{\rho}(y)(a,\varphi\alpha^{-1}(x))+\rho(x)(a,\varphi\alpha^{-1}(y)),
 \end{eqnarray*}
which yields that Eq.~~(\ref{E4}) holds. Analogously,
Eqs.~~(\ref{E2}) and (\ref{E6}) hold. Hence, the non-abelian 3-cocycles
$(\omega_{(\alpha,\beta)},\theta_{(\alpha,\beta)},\rho_{(\alpha,\beta)})$
and $(\omega,\theta,\rho)$ are equivalent via the linear map
$\varphi\alpha^{-1}:\mathfrak g\longrightarrow \mathfrak h$.

The converse part can be checked similarly.

\end{proof}
%%%%%%%%%%%%%%%%%%%%%%%%%%%%%%%%%%%%%%%%%%%%%%%%%%%%%%%%%%%%%%%%%%%%%%%%%%%%%%%%%%%%%%%%%%%%%%%%%%%%%%%%%%
\section{Wells exact sequences for Lie triple systems}
In this section, we are focus on the Wells map associated with non-abelian extensions of Lie triple systems.

Let $\mathcal{E}:0\longrightarrow\mathfrak h\stackrel{i}{\longrightarrow}
\hat{\mathfrak g}\stackrel{p}{\longrightarrow}\mathfrak
g\longrightarrow0$
be a non-abelian extension of $\mathfrak
g$ by $\mathfrak
h$ with a section $s$ of $p$.
Suppose that $(\omega,\theta,\rho)$ is the corresponding non-abelian 3-cocycle
induced by $s$.
Define a linear map $W:\mathrm{Aut}(\mathfrak
g)\times \mathrm{Aut}(\mathfrak
h)\longrightarrow H^{3}_{nab}(\mathfrak
g,\mathfrak
h)$ by
\begin{equation}\label{W1}
	W(\alpha,\beta)=[(\omega_{(\alpha,\beta)},\theta_{(\alpha,\beta)},
\rho_{(\alpha,\beta)})
-(\omega,\theta,\rho)].
\end{equation}
The map $W$ is called the Wells map. In view of  Lemma \ref{Le1},
the Wells map $W$ does not depend on the choice of sections. 

 \begin{thm} \label{Wm3} Assume that
$\mathcal{E}:0\longrightarrow\mathfrak h\stackrel{i}{\longrightarrow}
\hat{\mathfrak g}\stackrel{p}{\longrightarrow}\mathfrak g\longrightarrow0$
is a non-abelian extension of $\mathfrak g$ by $\mathfrak h$ with a section $s$ of $p$.
Then there is an exact sequence:
$$1\longrightarrow \mathrm{Aut}_{\mathfrak h}^{\mathfrak g}(\hat{\mathfrak g})\stackrel{H}{\longrightarrow} \mathrm{Aut}_{\mathfrak h}(\hat{\mathfrak g})\stackrel{\lambda}{\longrightarrow}\mathrm{Aut}(\mathfrak g)\times \mathrm{Aut}(\mathfrak h)\stackrel{W}{\longrightarrow} H^{3}_{nab}(\mathfrak g,\mathfrak h),$$
where $\mathrm{Aut}_{\mathfrak h}^{\mathfrak g}(\hat{\mathfrak g})=\{\gamma \in \mathrm{Aut}(\hat{\mathfrak g})| \lambda(\gamma)=(I_{\mathfrak g},I_{\mathfrak h}) \}$.\end{thm}

\begin{proof} Obviously, $\mathrm{Ker} \lambda=\mathrm{Im}H$ and $H$ is injective.
By Theorem \ref{Eth1}, one can easily check that $\mathrm{Ker} W=\mathrm{Im}\lambda$. This completes the proof.
\end{proof}
Let
$\mathcal{E}:0\longrightarrow\mathfrak h\stackrel{i}{\longrightarrow}
\hat{\mathfrak g}\stackrel{p}{\longrightarrow}\mathfrak g\longrightarrow0$
be a non-abelian extension of $\mathfrak g$ by $\mathfrak h$  with a section $s$ of $p$. Assume that
$(\omega,\theta,\rho)$ is the non-abelian 3-cocycle corresponding to $\mathcal{E}$ induced by $s$.
Denote
\begin{align}
		Z_{nab}^{1}(\mathfrak g,\mathfrak h)=&\left\{\varphi:\mathfrak g\rightarrow \mathfrak h\left|
\begin{aligned}&
    [a,\varphi(x),b]_{\mathfrak
     h}=0,
     \\&[a,\varphi(x),\varphi(y)]_{\mathfrak h}-D_{\rho}(y)(a,\varphi(x))+\rho(x)(a,\varphi(y))=0,
    \\&D_{\rho}(z)(\varphi(x),\varphi(y))-\rho(y)(\varphi(x),\varphi(z))
     -\theta(y,z)\varphi(x)\\&+\rho(x)(\varphi(y),\varphi(z))+\theta(x,z)\varphi(y)
-D_{\theta}(x,y)\varphi(z)\\&=[\varphi(x),\varphi(y),\varphi(z)]_{\mathfrak
     h}-\varphi([x,y,z]_{\mathfrak g}),~\forall~x,y,z\in {\mathfrak g},a,b\in {\mathfrak h}
     \end{aligned}\right.\right\}.\label{W5}
	\end{align}
It is easy to check that $Z_{nab}^{1}(\mathfrak g,\mathfrak h)$ is an abelian group, which is called a non-abelian 1-cocycle on $\mathfrak g$ with values in $\mathfrak h$.
  \begin{pro} \label{Wm4} With the above notations, we have

(i) The linear map $K:\mathrm{Ker} \lambda\longrightarrow Z_{nab}^{1}(\mathfrak g,\mathfrak h)$ given by
 \begin{equation}\label{W6}K(\gamma)(x)=\varphi_{\gamma}(x)=s(x)-\gamma s(x),~\forall~~\gamma\in \mathrm{Ker} \lambda,x\in \mathfrak g \end{equation}
  is a homomorphism of groups.

(ii) $K$ is an isomorphism, that is, $\mathrm{Ker}\lambda\simeq Z_{nab}^{1}(\mathfrak g,\mathfrak h)$.

\end{pro}

\begin{proof}
(i)
	Using Eqs.~(\ref{C1})-(\ref{C3}) and (\ref{W6}), for all $x,y,z\in \mathfrak g$, we have,
\begin{align*}&[\varphi_{\gamma}(x),\varphi_{\gamma}(y),\varphi_{\gamma}(z)]_{\mathfrak h}-D_{\rho}(z)(\varphi_{\gamma}(x),\varphi_{\gamma}(y))+\rho(y)(\varphi_{\gamma}(x),\varphi_{\gamma}(z))+
\theta(y,z)\varphi_{\gamma}(x)\\&-\rho(x)(\varphi_{\gamma}(y),\varphi_{\gamma}(z))-\theta(x,z)\varphi_{\gamma}(y)
+D_{\theta}(x,y)\varphi_{\gamma}(z)-\varphi_{\gamma}(\{x,y,z\}_{\mathfrak g})
\\=&[s(x)-\gamma s(x),s(y)-\gamma s(y),s(z)-\gamma s(z)]_{\hat{\mathfrak g}}-[s(x)-\gamma s(x),s(y)-\gamma s(y),s(z)]_{\hat{\mathfrak g}}
\\&+[s(y),s(x)-\gamma s(x),s(z)-\gamma s(z)]_{\hat{\mathfrak g}}+[s(x)-\gamma s(x),s(y),s(z)]_{\hat{\mathfrak g}}
-[s(x),s(y)-\gamma s(y),s(z)-\gamma s(z)]_{\hat{\mathfrak g}}\\&-[s(y)-\gamma s(y),s(x),s(z)]_{\hat{\mathfrak g}}
+[s(x),s(y),s(z)-\gamma s(z)]_{\hat{\mathfrak g}}+\gamma s([x,y,z]_{\mathfrak g})-s([x,y,z]_{\mathfrak g})
\\=&\gamma s([x,y,z]_{\mathfrak g})-[\gamma s(x),\gamma s(y),\gamma s(z)]_{\hat{\mathfrak g}}+[s(x),s(y),s(z)]_{\hat{\mathfrak g}}-s([x,y,z]_{\mathfrak g})
\\=&\omega(x, y,z)-\gamma \omega(x, y,z)
\\=&0.
\end{align*}
Analogously, we can check that $\varphi_{\gamma}$ satisfies the other identities in $ Z_{nab}^{1}(\mathfrak g,\mathfrak h)$.
Thus, $K$ is well-defined.
For any $\gamma_1,\gamma_2\in \mathrm{Ker} \lambda$ and $x\in \mathfrak g$, suppose $K(\gamma_1)=\varphi_{\gamma_1}$ and $K(\gamma_2)=\varphi_{\gamma_2}$.
By Eq.~~ (\ref{W6}), we have
\begin{align*}K(\gamma_1 \gamma_2)(x)&=s(x)-\gamma_1 \gamma_2s(x)
\\&=s(x)-\gamma_1(s(x)-\varphi_{\gamma_2}(x))
\\&=s(x)-\gamma_1s(x)+\gamma_{1}\varphi_{\gamma_2}(x)
\\&=\varphi_{\gamma_1}(x)+\varphi_{\gamma_2}(x)\end{align*}
which means that $K(\gamma_1 \gamma_2)=K(\gamma_1)+K( \gamma_2)$ is a homomorphism of groups.

(ii) For all $\gamma\in \mathrm{Ker}\lambda$, we obtain that $\lambda(\gamma)=(p\gamma s,\gamma|_{\mathfrak h})=(I_\mathfrak g,I_\mathfrak h)$.
 If $K(\gamma)=\varphi_{\gamma}=0$,
we can get $\varphi_{\gamma}(x)=s(x)-\gamma s(x)=0$, that is, $\gamma=I_{\hat{\mathfrak g}}$, which indicates that $K$ is injective.
Secondly, we prove that $K$ is surjective. Since all $\hat{x}\in \hat{\mathfrak g}$ can be written as $a+s(x)$ for
 some $a\in \mathfrak h, x\in \mathfrak g$, for any $\varphi\in Z_{nab}^{1}(\mathfrak g,\mathfrak h)$, we
 define a linear map $\gamma:\hat{\mathfrak g}\rightarrow \hat{\mathfrak g}$ by
  \begin{equation}\label{W7}\gamma(\hat{x})=\gamma(a+s(x))=s(x)-\varphi(x)+a,~\forall~\hat{x}\in \hat{\mathfrak g}.\end{equation}
It is obviously that $(p\gamma s,\gamma|_{\mathfrak h})=(I_\mathfrak g,I_\mathfrak h)$.
We need to verify that $\gamma $ is an automorphism of Lie triple system $\hat{\mathfrak g}$. One can take the same
procedure as the proof of the converse part of Theorem \ref{EC}.
 It follows that $\gamma\in \mathrm{Ker} \lambda$.
Thus, $K$ is surjective. In all, $K$ is bijective.
 So $\mathrm{Ker }\lambda\simeq Z_{nab}^{1}(\mathfrak g,\mathfrak h)$.

\end{proof}

Combining Theorem \ref{Wm3} and Proposition \ref{Wm4}, we have

\begin{thm} \label{Wm5} Let
$\mathcal{E}:0\longrightarrow\mathfrak h\stackrel{i}{\longrightarrow}
\hat{\mathfrak g}\stackrel{p}{\longrightarrow}\mathfrak g\longrightarrow0$
be a non-abelian extension of $\mathfrak g$ by $\mathfrak h$. There is an exact sequence:
$$0\longrightarrow Z_{nab}^{1}(\mathfrak g,\mathfrak h)\stackrel{i}{\longrightarrow} \mathrm{Aut}_{\mathfrak h}(\hat{\mathfrak g})\stackrel{\lambda}{\longrightarrow}\mathrm{Aut}(\mathfrak g)\times \mathrm{Aut}(\mathfrak h)\stackrel{W}{\longrightarrow} H^{3}_{nab}(\mathfrak g,\mathfrak h).$$
\end{thm}

%%%%%%%%%%%%%%%%%%%%%%%%%%%%%%%%%%%%%%%%%%%%%%%%%%%%%%%%%%%%%%%%%%%%%%%%%%%%%%%%%%%%%%%%%%%%%%%%%%%%%%%
\section{Particular case: abelian extensions of Lie triple systems}
 In this section, we interpret the results of previous section in particular case.

Let
$\mathcal{E}:0\longrightarrow\mathfrak h\stackrel{i}{\longrightarrow} \hat{\mathfrak g}\stackrel{p}{\longrightarrow}\mathfrak g\longrightarrow0$
be an abelian extension of $\mathfrak g$ by $\mathfrak h$ with a section $s$ of $p$.
Suupose that $\omega$ is a 3-cocycle corresponding to $\mathcal{E}$ induced by $s$.
We have,
\begin{thm}[\cite{35}]
	(i) The triple $(\mathfrak g\oplus \mathfrak h,[\ , \ , \  ]_{(\omega,\theta)})$ is a Lie triple system
 if and only if $\omega$ is a 3-cocycle of $\mathfrak g$ with coefficients in the
representation $(\mathfrak h,\theta)$.

(ii) Abelian extensions of $\mathfrak g$
 by $\mathfrak h$ are classified
by the cohomology group $H^{3}(\mathfrak g,\mathfrak h)$ of
$\mathfrak g$ with coefficients in $(\mathfrak h,\theta)$.

\end{thm}

The following statement can be got directly by Theorem \ref{EC}.
\begin{thm} Let $\mathcal{E}:0\longrightarrow\mathfrak h\stackrel{i}{\longrightarrow} \hat{\mathfrak g}
\stackrel{p}{\longrightarrow}\mathfrak g\longrightarrow0$ be an abelian extension of $\mathfrak g$
by $\mathfrak h$ with a section $s$ of $p$. Assume that $\omega$ is a 3-cocycle
and $(\mathfrak h,\theta)$ is a representation of $\mathfrak g$
associated to $\mathcal{E}$. A pair $(\alpha,\beta)\in \mathrm{Aut}(\mathfrak g
)\times \mathrm{Aut}(\mathfrak h)$ is inducible with respect to the abelian extension $\mathcal{E}$ if and only if there is a
linear map $\varphi:\mathfrak g\longrightarrow \mathfrak h$
satisfying the following conditions:
\begin{align}
     \beta\omega(x,y,z)-\omega(\alpha(x),\alpha(y),\alpha(z))=&\nonumber
\theta(\alpha(x),\alpha(z))\varphi(y)-\theta(\alpha(y),\alpha(z))\varphi(x)
\\&\label{AEE1}-D_{\theta}(\alpha(x),\alpha(y))\varphi(z)+\varphi([x,y,z]_{\mathfrak
     g}),
\end{align}
\begin{equation}\label{AEE2}
     \beta(\theta(x,y)a)=\theta(\alpha(x),\alpha(y))\beta(a) .
\end{equation}
\end{thm}

According to Eqs.~(\ref{RLts3}) and (\ref{AEE2}), we have
\begin{equation*}\beta D_{\theta}(x,y)a=D_{\theta}(\alpha(x),\alpha(y))\beta(a).
\end{equation*}

For all $(\alpha,\beta)\in \mathrm{Aut}(\mathfrak g
)\times \mathrm{Aut}(\mathfrak h)$, $\omega_{(\alpha,\beta)}$ may not be a 3-cocycle.
 Indeed,  $\omega_{(\alpha,\beta)}$ is a 3-cocycle if Eq.~(\ref{AEE2}) holds.
 Thus, it is natural to introduce the space of compatible pairs of automorphisms:
\begin{align*}
		C_{\theta}=&\left\{(\alpha,\beta)\in \mathrm{Aut}(\mathfrak g
)\times \mathrm{Aut}(\mathfrak h)\left|\begin{aligned}&\beta(\theta(x,y)a)=\theta(\alpha(x),\alpha(y))\beta(a),
~\forall~x,y\in {\mathfrak g},a\in {\mathfrak h}
     \end{aligned}\right.\right\}.
	\end{align*}

Analogous to Theorem \ref{Eth1}, we have

\begin{thm} \label{Wm6} Let $\mathcal{E}:0\longrightarrow\mathfrak h\stackrel{i}{\longrightarrow} \hat{\mathfrak g}
\stackrel{p}{\longrightarrow}\mathfrak g\longrightarrow0$ be an abelian extension of $\mathfrak g$
by $\mathfrak h$ with a section $s$ of $p$ and $\omega$ be a
3-cocycle associated to $\mathcal{E}$ induced by $s$. A pair $(\alpha,\beta)\in C_{\theta}$ is inducible with
  respect to the abelian extension
$\mathcal{E}$ if and only if  $\omega$ and $\omega_{(\alpha,\beta)}$ are in the same cohomological class.
\end{thm}

In the case of abelian extensions, $Z^{1}_{nab}(\mathfrak g,\mathfrak h)$
defined by (\ref{W5}) becomes to ${H}^{1}(\mathfrak g,\mathfrak h)$ given in Section 2.
 In the light of Theorem \ref{Wm5} and Theorem \ref{Wm6}, we have the following exact sequence:

\begin{thm} Let $\mathcal{E}:0\longrightarrow\mathfrak h\stackrel{i}{\longrightarrow} \hat{\mathfrak g}
\stackrel{p}{\longrightarrow}\mathfrak g\longrightarrow0$ be an abelian extension of $\mathfrak g$
by $\mathfrak h$. There is an exact sequence:
$$0\longrightarrow H^{1}(\mathfrak g,\mathfrak h)\stackrel{i}{\longrightarrow} \mathrm{Aut}_{\mathfrak h}(\hat{\mathfrak g})\stackrel{\lambda}{\longrightarrow}C_{\theta}\stackrel{W}{\longrightarrow} H^{3}(\mathfrak g,\mathfrak h).$$
\end{thm}

%%%%%%%%%%%%%%%%%%%%%%%%%%%%%%%%%%%%%%%%%%%%%%%%%%%%%%%%%%%%%%%%%%%%%%%%%%%%%%%%%%%%%%%%%%%%%%%%%%%%%%%%%%

\begin{center}{\textbf{Acknowledgments}}
\end{center}
The first author is supported by the NSF of
China ( No. 11871421), the Natural Science Foundation of Zhejiang
Province of China ( No. LY19A010001) and the Science and Technology
Planning Project of Zhejiang Province (No. 2022C01118). The second author is supported by the NSF of China (No. 12161013) and    Guizhou Provincial Basic Research Program (Natural Science) (No. ZK[2023]025).

%%%%%%%%%%%%%%%%%%%%%%%%%%%%%%%%%%%%%%%%%%%%%%%%%%%%%%%%%%%%%%%%%%%%%%%%%%%%%%%%%%%%%%%%%%%%%%%%%%%%%%%%%%
\begin{center} {\textbf{Statements and Declarations}}
\end{center}
 All datasets underlying the conclusions of the paper are available
to readers. No conflict of interest exits in the submission of this
manuscript.

%%%%%%%%%%%%%%%%%%%%%%%%%%%%%%%%%%%%%%%%%%%%%%%%%%%%%%%%%%%%%%%%%%%%%%%%%%%%%%%%%%%%%%%%%%%%%%%%%%%%%%%%%%


\begin{thebibliography} {99}

\bibitem [1] {60} V. G. Bardakov, M. Singh, Extensions and automorphisms of Lie algebras, J. Algebra Appl. 16 (2017), 1750162.

\bibitem [2] {00} S. Chen, Q. Lou, Q. Sun, Cohomologies of Rota-Baxter Lie triple
systems and applications, Commun. Algebra 51 (10) (2023), 4299-4315.

\bibitem [3] {20} T. Chtioui, A. Hajjaji, S. Mabrouk, A. Makhlouf, Cohomologies
and deformations of $\mathcal O$-operators on Lie triple systems, J.
Math. Phys. 64, 081701 (2023).

\bibitem [4] {004} A. Das, N. Ratheeb, Extensions and automorphisms of Rota-Baxter groups, J. Algebra 636 (2023), 626–665.

\bibitem [5] {010} S. Eilenberg, S. Maclane, Cohomology theory in abstract groups. II.
Group extensions with non-abelian kernel, Ann. Math. 48 (1947), 326-341.


\bibitem [6] {048} Y. Fr{\'e}gier, Non-abelian cohomology of extensions of Lie algebras as Deligne groupoid, J. Algebra 398 (2014), 243–257.

\bibitem [7] {021} S. Goswamia, S. K. Mishraa, G. Mukherjee, Automorphisms of extensions of Lie-Yamaguti algebras and inducibility problem,
J. Algebra 641 (2024), 268–306.

\bibitem [8] {013} J. Gouray, A differential graded Lie algebra approach to non-abelian extensions of associative algebras, arXiv:1802.04641.

\bibitem [9] {08} Y. Guo, B. Hou, Crossed modules and non-abelian extensions of Rota-Baxter Leibniz
algebras, J. Geom.Phys. 191 (2023), 104906.

\bibitem[10] {026} S. K. Hazra, A. Habib, Wells exact sequence and automorphisms of extensions of Lie superalgebras, J. Lie Theory 30 (2020), 179–199.

\bibitem [11] {21} B. Harris, Cohomology of Lie triple systems and Lie algebras with
involution, Trans. Amer. Math. Soc. 98 (1961), 148-162.

\bibitem [12]{22} L. Hodge, J. Parshall, On the representation theory
of Lie triple systems, Trans. Amer. Math. Soc. 354 (2002), 4359-4391.

\bibitem [13] {047} B. Hou, J. Zhao, Crossed modules, non-abelian extensions of associative conformal algebras and Wells exact sequences, arXiv:2211.10842v1.

\bibitem [14] {027} N. Inassaridze, E. Khmaladze, M. Ladra, Non-abelian cohomology and extensions of Lie algebras, J. Lie Theory 18 (2008), 413–432.

\bibitem [15]{51} N. Jacobson, Lie and Jordan triple systems, Amer. J. Math.
71 (1949), 149-170.

\bibitem [16] {045} P. Jin, H. Liu, The Wells exact sequence for the automorphism group of a group extension, J. Algebra 324 (2010), 1219–1228.

\bibitem [17] {30} F. Kubo, Y. Taniguchi, A controlling cohomology of the
deformation theory of Lie triple systems, J. Algebra 278 (2004), 242-250.

\bibitem [18] {09} S. K. Mishra, A. Das, S. K. Hazra, Non-abelian extensions of Rota-Baxter Lie algebras and inducibility
of automorphisms, Linear. Algebra. Appl. 669 (2023), 147-174.

\bibitem [19]{27} W. Lister, A structure theory for Lie triple systems, Trans.
Amer. Math. Soc. 72 (1952), 217-242.

 \bibitem [20] {029} J. Liu, Y. Sheng, Q. Wang, On non-abelian extensions of Leibniz algebras, Commun. Algebra 46 (2) (2018), 574–587.

\bibitem [21] {201} Q. Ma, L. Song, Y. Wang, Non-abelian extensions of pre-Lie algebras, Commun. Algebra
51 (4) (2023), 1370-1382.

\bibitem [22] {54} Y. Ma, L. Chen, J. Lin, Systems of quotients of Lie triple
systems, Commun. Algebra 42 (2014), 3339-3349.

\bibitem [23] {014} N. Machado, J. Oinert, S. Wagner, Non-abelian extensions of
groupoids and their groupoid rings, arXiv:2207.03369.

\bibitem [24] {015} K. H. Neeb, Non-abelian extensions of infinite-dimensional Lie
groups, Ann. Inst. Fourier (Grenoble) 57 (2007), 209-271.

\bibitem [25] {032} I. B. S. Passi, M. Singh, M. K. Yadav, Automorphisms of abelian group extensions, J. Algebra 324 (2010), 820–830.

\bibitem [26] {016} L. Song, A. Makhlouf, R. Tang, On non-abelian extensions of 3-Lie
algebras, Commun. Theor. Phys. 69 (2018), 347-356.

\bibitem [27] {67} C. Wells, Automorphisms of group extensions, Trans. Amer. Math. Soc. 155 (1971), 189-194.

\bibitem [28] {53} R. Yadav, N. Behera, R. Bhutia, Equivariant one-parameter
deformations of Lie triple systems, J. Algebra 568 (2021), 467-479.

\bibitem [29]{34} K. Yamaguti, On the cohomology space of Lie triple system, Kumamoto J.
Sci. Ser. A 5 (1960), 44-52.

\bibitem [30]{135} K. Yamaguti, On the Lie triple system and its generalization, J. Sci. Hiroshima
Univ. Ser. A 21, (1957-1958), 155-159.

\bibitem [31] {35} T. Zhang, Notes on cohomologies of Lie triple systems, J. Lie Theory 24(4)(2014), 909-929.




















\end{thebibliography}
\end {document}